\documentclass[preprint,10pt,3p]{elsarticle}
\usepackage{amsmath,amsthm,amsfonts,amssymb}
\usepackage{xcolor}
\usepackage{algorithm}
\usepackage{multirow}
\usepackage{multicol}
\usepackage{booktabs}% for better rules in the table
\usepackage{algorithm}               %format of the algorithm
\usepackage{algorithmic}
\usepackage{comment}
\usepackage{threeparttable}

\usepackage{relsize}

\usepackage{calligra}
\usepackage{carolmin}
\usepackage[T1]{fontenc}

\usepackage{auncial}
\usepackage[B1]{fontenc}

\usepackage{mathrsfs}
\usepackage{sectsty}
\definecolor{astral}{RGB}{46,116,181}
\sectionfont{\color{astral}}
\newtheorem{theorem}{Theorem}[section]
\newtheorem{lemma}[theorem]{Lemma}
\newtheorem{corollary}[theorem]{Corollary}

\theoremstyle{definition}
\newtheorem{definition}[theorem]{Definition}
\newtheorem{example}[theorem]{Example}

\usepackage{hyperref}
\usepackage{subfigure}

\usepackage{tikz}
\pgfdeclarelayer{bg}
\pgfsetlayers{bg,main}

\usepackage{tikz,xcolor}
\definecolor{lime}{HTML}{A6CE39}
\definecolor{lightblue}{rgb}{0.0, 0.0, 0.5}
\DeclareRobustCommand{\orcidicon}{%
	\begin{tikzpicture}
	\draw[lime, fill=lime] (0,0)
	circle [radius=0.16]
	node[white] {{\fontfamily{qag}\selectfont \tiny ID}};
	\draw[white, fill=white] (-0.0625,0.095)
	circle [radius=0.007];
	\end{tikzpicture}
	\hspace{-2mm}
}
\foreach \x in {A, ..., Z}{%
	\expandafter\xdef\csname orcid\x\endcsname{\noexpand\href{https://orcid.org/\csname orcidauthor\x\endcsname}{\noexpand\orcidicon}}
}

\usepackage{wasysym}

\usepackage{tikz}
\pgfdeclarelayer{bg}
\pgfsetlayers{bg,main}

\usepackage{tikz,xcolor}
\definecolor{lime}{HTML}{A6CE39}
\definecolor{lightblue}{rgb}{0.0, 0.0, 0.5}
\DeclareRobustCommand{\orcidicon}{%
	\begin{tikzpicture}
	\draw[lime, fill=lime] (0,0)
	circle [radius=0.16]
	node[white] {{\fontfamily{qag}\selectfont \tiny ID}};
	\draw[white, fill=white] (-0.0625,0.095)
	circle [radius=0.007];
	\end{tikzpicture}
	\hspace{-2mm}

}
\newcommand{\rra}[1]{\mathrm{rshrank}(#1)}  % rank

\newcommand{\bigrho}{\raisebox{-.0\baselineskip}{\Large\ensuremath{\rho}}}
\newcommand{\bigAlpha}{\raisebox{-.0\baselineskip}{\large\ensuremath{\alpha}}}
\newcommand{\bigbBeta}{\raisebox{-.0\baselineskip}{\large\ensuremath{\beta}}}
\newcommand{\bigDrho}{\bigrho_2}
\newcommand{\bigBrho}{\bigrho_{_{B,*}}}
\newcommand{\bigCrho}{\bigrho_{_{*,C}}}
\newcommand{\bigBCrho}{\bigrho_{_{B,C}}}
\newcommand{\bigdelta}{\raisebox{-.1\baselineskip}{\Large\ensuremath{\delta}}}
\newcommand{\bigDdelta}{\bigdelta_2}
\newcommand{\bigBdelta}{\bigdelta_{_{B,*}}}
\newcommand{\bigCdelta}{\bigdelta_{_{*,C}}}
\newcommand{\bigBCdelta}{\bigdelta_{_{B,C}}}

\numberwithin{equation}{section}

\begin{document}
\begin{frontmatter}
\title{%$M$-product based
Perturbation analysis of tensor $(\mathcal{B},\mathcal{C})$-inverse via Einstein product
}
\author{Daochang Zhang$^{a,1}$, Jingqian Li$^{a,2}$, Dijana Mosi\'c$^{b,3}$, Predrag S. Stanimirovi\' c$^{b,4}$}
\address{
 $^{a}$ College of Sciences, Northeast Electric Power University, Jilin, P.R. China.\\
$^{b}$ Faculty of Sciences and Mathematics, University of Ni\v s, 18000 Ni\v s, Serbia.\\
 \textit{E-mail}$^{1}$: \texttt{daochangzhang@126.com}\\
\textit{E-mail}$^{2}$: \texttt{jingqianli85@163.com}\\
\textit{E-mail}$^{3}$: \texttt{dijana@pmf.ni.ac.rs}\\
\textit{E-mail}$^{4}$: \texttt{pecko@pmf.ni.ac.rs}
}

\begin{abstract}
We investigate the influence of a relatively small perturbation on various generalized inverses functions or quantities derived from a tensor $\mathcal{A}$.
When a small tensor perturbation \(\mathcal{E}\) is introduced, it becomes challenging to analyze generalized inverses of the perturbed tensor \( \mathcal{D} =\mathcal{A}+\mathcal{E}\) and to determine how this perturbation affects a generalized inverse of $\mathcal{A}$.
Our main goal is to understand the relationship between $\mathcal{D}^\Game$ and \( \mathcal{A}^\Game \), where $(\cdot)^\Game$ denotes a specific generalized inverse or a class of generalized inverses.
 In particular, classes of tensor inner, outer, and $(\mathcal{B},\mathcal{C})$ inverses are considered.

\end{abstract}

\end{frontmatter}

\section{Introduction}

Tensors, as powerful higher-order generalizations of matrices, play a significant role in linear least squares problems, statistical applications, geology, computer vision, and in physics where they effectively describe systems with multiple interactions.
They are also integral to engineering for analyzing various physical phenomena, and data processing \cite{SJMAA344,TMS871,SEG762,FM122,SR512,MRI296}.
Fundamental structures and properties of tensors have been thoroughly explored by researchers, highlighting their significance and versatility across various applications \cite{NLAA293,JAMC682,LAIA4393,SR512,LMA673,UOTW221,AEOTM182}.

Perturbation in tensor calculus refers to examining how a small change in a given tensor affects functions or quantities derived from it.
If we perturb a tensor \(\mathcal{A} \) by a small tensor \(\mathcal{E} \), we consider \( \mathcal{D} =\mathcal{A}+\mathcal{E}\) and analyze how this perturbation affects to a generalized inverse of $\mathcal{A}$.
More precisely, we want to understand the relationship between \( \mathcal{D}^\Game \) and \( \mathcal{A}^\Game \), where $(\cdot)^\Game$ denotes a generalized inverse under consideration.

The following topics have been investigated in relation to the perturbation analysis.

Continuity of generalized inverses.
Under certain conditions, the generalized inverse can vary continuously concerning small perturbations in the tensor.
This continuity implies that small changes in \( \mathcal{A} \) result in small changes in \( \mathcal{A}^\Game \).

Perturbation results.
If \(\mathcal{E}\) is sufficiently small, there are bounds that relate the distance between \(\mathcal{A}^\Game\) and \((\mathcal{A} + \mathcal{E})^\Game\) in terms of the norm of the perturbation \(\mathcal{E}\).

Understanding the effects of perturbations is essential in numerical analysis, data fitting, and control theory, where accurate calculations of generalized inverses can lead to significant differences in outcomes.

Operations on tensors, such as tensor addition, multiplication, and contraction, extend those defined for matrices to accommodate these higher dimensions.
Tensor operations include the Kronecker product, Hadamard product, T-product, and Einstein product \cite{PUPP071,SR512}.
In particular, the research on the Einstein product due to its notable applications in physics.

As the operation for tensors, the Einstein product $*_n $ was introduced by Einstein \cite{PUPP071} and expressed as follows
\begin{equation*}
(\mathcal{A}*_n\mathcal{B})_{i_1\cdot\cdot\cdot i_kj_1\cdot\cdot\cdot j_m}=\sum_{h_1\cdot\cdot\cdot h_{n^\prime}}a_{i_1\cdot\cdot\cdot i_kh_1\cdot\cdot\cdot h_{n}}b_{h_1\cdot\cdot\cdot h_{n} j_1\cdot\cdot\cdot j_m}.
\end{equation*}
In this formulation, tensors $\mathcal{A}$ and $\mathcal{B}$ satisfy $\mathcal{A}\in\mathbb{C}^{I_1\times \cdot\cdot\cdot \times I_k\times H_1\times \cdot\cdot\cdot \times H_{n}}$ and
$\mathcal{B}\in\mathbb{C}^{H_1\times \cdot\cdot\cdot \times H_{n}\times J_1\times \cdot\cdot\cdot \times J_m}$.

Generalized inverse theory occupies a central position in tensor algebra. Scholars have studied various tensor inverses including the Moore-Penrose inverse, Drazin inverse, group inverse, and $(\mathcal{B},\mathcal{C})$-inverse. The $(\mathcal{B},\mathcal{C})$-inverse, as an extension of outer inverses, demonstrates unique advantages in solving Poisson equations and color image denoising problems due to its specific algebraic conditions. In 2018, Stanimirov\' c et al. formally defined the $(\mathcal{B},\mathcal{C})$-inverse under the Einstein product \cite{LMA685}.
In 2020, Ke et al. analyzed one-sided $(\mathcal{B},\mathcal{C})$-inverses over rings \cite{FOSIR343}.
Drazin gave the exploration of the weighted $(\mathcal{B},\mathcal{C})$-inverse properties in semigroups \cite{CIA481}.
In 2021, Zhu {et al.} investigated $(\mathcal{B},\mathcal{C})$-inverse in rings \cite{RRACE1153}.
Wu { et al.} represented the $(\mathcal{B},\mathcal{C})$-inverse and $(\mathcal{C},\mathcal{B})$-inverse, through the group inverse \cite{CIA492}.
Furthermore, Mosi\'c { et al.} gave the algorithm for computing the weighted tensor $(\mathcal{B},\mathcal{C})$-inverse \cite{JCAM3883}.

The Moore-Penrose inverse $\mathcal{X}\in \mathbb{C}^{T(t) \times S(s)}$ of any $\mathcal{A}\in \mathbb{C}^{S(s) \times T(t)}$, denoted as $\mathcal{A}^\dag$, satisfies four specific conditions \cite{LMA643}:
\begin{center}
$(1^T) \ \mathcal{A} \, {*_t} \, \mathcal{X} \, {*_s} \, \mathcal{A} = \mathcal{A}$,\
$(2^T)\  \mathcal{X} \, {*_s} \, \mathcal{A} \, {*_t} \, \mathcal{X} = \mathcal{X}$,\
$(3^T) \ (\mathcal{A} \, {*_t} \, \mathcal{X})^* = \mathcal{A} \, {*_t} \, \mathcal{X}$,\
$(4^T)\ (\mathcal{X} \, {*_s} \, \mathcal{A})^* = \mathcal{X} \, {*_s} \, \mathcal{A}$.
\end{center}

If \( \mathcal{X} \) fulfills condition \((2^T)\), it is an outer inverse of \( \mathcal{A} \).
The symbol \( \mathcal{A}_{\aunclfamily{\bf R}(\mathcal{B}), *}^{(2)} \) (resp. \( \mathcal{A}_{*, \aunclfamily{\bf N}(\mathcal{C})}^{(2)} \)) denotes an outer inverse \(\mathcal{X}\) of \(\mathcal{A}\)
which has a prescribed range \( \aunclfamily{\bf R}(\mathcal{X}) = \aunclfamily{\bf R}(\mathcal{B}) \) (resp. null space \(\aunclfamily{\bf N}(\mathcal{X}) = \aunclfamily{\bf N}(\mathcal{C}) \)).
Furthermore, \( \mathcal{A}_{\aunclfamily{\bf R}(\mathcal{B}), \aunclfamily{\bf N}(\mathcal{C})}^{(2)} \) stands for an arbitrary outer inverse \( \mathcal{X} \) of \( \mathcal{A} \) fulfilling constraints \( \aunclfamily{\bf R}(\mathcal{X}) = \aunclfamily{\bf R}(\mathcal{B}) \) and \( \aunclfamily{\bf N}(\mathcal{X}) = \aunclfamily{\bf N}(\mathcal{C}) \).

Perturbation theory for the Drazin inverse of matrices was initiated by Wei {et al.} in 1997 \cite{LANIA2582}.
Additionally, the relevant research on the Drazin inverse can be found in \cite{ZMT-LAA,DZYZDK,JCAMDYD}.
Subsequent studies expanded the perturbation analysis to various generalized inverses in tensor domain.
In 2019, Ma {et al.} investigated perturbations of the Moore-Penrose inverse for tensors under the Einstein product \cite{CAM383}.
Du {et al.} studied additive perturbations for the tensor core inverse under the Einstein product and provided an upper bound for the perturbation \cite{FPTF333}.
In 2021, Mosi\' c {et al.} investigated the MPCEP inverse perturbations of matrices \cite{JAMC673}.
In addition, Cui {et al.} explored the Moore-Penrose inverse perturbations of the tensors under the T-product \cite{FTPBF352}.
Despite its significance, perturbation analysis for the $(\mathcal{B}, \mathcal{C})$-inverse has not yet been sufficiently developed.
Addressing this gap could lead to significant advancements in the field.

For convenience, we restate the notation
\begin{center}
     $L(k)=L_1\times \cdots \times L_k$, %=\{L_1,\ldots, L_s|\ 1\leq l_j \leq L_j, j=1,\ldots,s\},$
\end{center}
where $L_1,\ldots,L_k$ are positive integers.
The tensor notation $\mathcal{A} =(a_{n_1,\ldots ,n_k})_{1\leq n_j\leq N_j}$ will be simplified by $\mathcal{A} = (a_{l(k)})$, where $j = 1,\ldots,k$.

The motivation of this paper is to investigate additive perturbations of the $(\mathcal{B},\mathcal{C})$-inverse of the tensors under the Einstein product.
The perturbation norm utilizes the consistency norm, which includes both the Frobenius norm and the spectral norm.
The Hermitian inner product of tensors \( \mathcal{A}, \mathcal{B} \in \mathbb{C}^{N(k)} \) is given in \cite{SJMAA344} by
\[
(\mathcal{A}, \mathcal{B}) = \sum_{l_1=1}^{L_1} \sum_{l_2=1}^{L_2} \cdots \sum_{l_k=1}^{L_k}  a_{l_1 l_2 \ldots l_k} \overline{b_{l_1 l_2 \ldots l_k}}.
\]
%In the case \( (\mathcal{A}, \mathcal{B}) = 0 \), \( \mathcal{A} \) is orthogonal to \( \mathcal{B} \).
The Frobenius norm of \( \mathcal{A} \) is defined as \( \Vert\mathcal{A}\Vert_F = \sqrt{(\mathcal{A}, \mathcal{A})} \).

If tensors \( \mathcal{A} \in \mathbb{C}^{T(t) \times T(t)} \) and \( \mathcal{X} \in \mathbb{C}^{T(t)} \) and a complex number \( \lambda \) fulfill
\[
\mathcal{A} *_t \mathcal{X} = \lambda \mathcal{X},
\]
the number \( \lambda \) is an eigenvalue of \( \mathcal{A} \) and \( \mathcal{X} \) is the eigentensor related to \( \lambda \) \cite{LMA673}.

The spectral norm of \( \mathcal{A}\in \mathbb{C}^{S(s) \times T(t)} \) is given by
\[
\|\mathcal{A}\|_2 = \sqrt{\lambda_{\max}(\mathcal{A}^* *_t \mathcal{A})} = \mu_1(\mathcal{A}),
\]
where \( \lambda_{\max}(\mathcal{A}^* *_t \mathcal{A}) \) is the largest eigenvalue of \( \mathcal{A}^* *_t \mathcal{A} \), and \( \mu_1(\mathcal{A}) \) is the largest singular value of \( \mathcal{A} \).

In \cite{LMA685}, the tensor reshaping operation transform was defined as
\(\text{rsh}\), and implemented by means of the standard MATLAB function \(\textit{reshape}\)
$\text{rsh}: \mathbb{C}^{S(s) \times T(t)} \mapsto \mathbb{C}^{\mathfrak{S} \times \mathfrak{T}},$
such that $\mathfrak{S} = \prod_{i=1}^{s} S_i$ and $\mathfrak{T} = \prod_{i=1}^{t} T_i$, as follows
\[
\text{rsh}(\mathcal{A}) = A = \text{reshape}(\mathcal{A}, \mathfrak{S}, \mathfrak{T}),\ \text{for}\ \mathcal{A} \in \mathbb{C}^{S(s)\times T(t)}\ \text{and}\
A \in \mathbb{C}^{\mathfrak{S} \times \mathfrak{T}}.
\]

Then the tensor {\em rshrank} (reshape rank) was defined in \cite{LMA685} following the following algorithm.
For positive integers $s,t$ and $S_1, \ldots, S_s, T_1, \ldots, T_t$, given tensor $\mathcal{A} \in \mathbb{C}^{S(s) \times T(t)}$, consider corresponding matrix $A = \textrm{rsh}(\mathcal{A}) \in \mathbb{C}^{\mathfrak{S} \times \mathfrak{T}}$.
If there exists $E \in \mathbb{C}^{\mathfrak{S} \times \mathfrak{S}}$ and $P \in \mathbb{C}^{\mathfrak{T} \times \mathfrak{T}}$ satisfying
\[ EAP = \begin{bmatrix} I_r & K \\ O & O \end{bmatrix}, \quad K \in \mathbb{C}^{r \times (\mathfrak{T}-r)}, \]
then $\rra{\mathcal{A}} = r$.

\smallskip
The research presented in this paper addresses the following perturbation problems.

{\bf Problem 1.}
Given a small perturbation $\mathcal{E}$ in $\mathcal{A}$, satisfying $\|\mathcal{A}^{(1)}\|\|\mathcal{E}\| < 1$.
Research question focuses on understanding how $\mathcal{E}$ affects the computation of $\mathcal{A}^{(1)}$.
Specifically, research task is to investigate relationships between $ (\mathcal{A} + \mathcal{E})^{(1)}$ and $\mathcal{A}^{(1)}$ defined by the multiplicative factors denoted as
$\bigrho =(\mathcal{I} + \mathcal{E}*_t \mathcal{A}^{(1)})^{-1}$ and $\bigdelta =(\mathcal{I} + \mathcal{A}^{(1)}*_s\mathcal{E})^{-1}$.
Essentially, research goal is to explore the relationship of the form $ (\mathcal{A} + \mathcal{E})^{(1)}=\mathcal{A}^{(1)} *_s\bigrho=\bigdelta *_t\mathcal{A}^{(1)}$ under specified conditions.

{\bf Problem 2.}
Given a small perturbation $\mathcal{E}$ in $\mathcal{A}$, satisfying the condition $\|\mathcal{E}*_t\mathcal{A}^{(2)}\|< 1$ along with some specified requirements.
Research question seeks to understand how $\mathcal{E}$ affects the computation of $\mathcal{A}^{(2)}$.
More specifically, the research task is to investigate relationship between $ (\mathcal{A} + \mathcal{E})^{(2)}$ and $\mathcal{A}^{(2)}$ defined by the multiplicative factors
$\bigDdelta   =(\mathcal{I} + \mathcal{A}^{(2)}*_s\mathcal{E})^{-1}$ and $\bigDrho  =(\mathcal{I} + \mathcal{E}*_t \mathcal{A}^{(2)})^{-1}$.
The main goal is to explore the relationship of the form $ (\mathcal{A} + \mathcal{E})^{(2)}=\mathcal{A}^{(2)} *_s\bigDrho =\bigDdelta   *_t\mathcal{A}^{(2)}$ under specified conditions.

{\bf Problem 3.} This research problem aims to generalize both {\bf Problem 1} and {\bf Problem 2} regarding outer inverses with prescribed ranges and/or null spaces.

\smallskip
The research is organized in the following sections.
Section 2 introduces the key lemmas and essential definitions.
Section 3 outlines equivalent conditions necessary for \( \mathcal{H} = \mathcal{A}^{(1)}*_s \bigrho \) to serve as an inner inverse of \( \mathcal{D} = \mathcal{A} + \mathcal{E} \), and additionally derives the inner inverse elated to the perturbation of the tensor.
Section 4 develops the perturbation computation and the error analysis for the $(\mathcal{B},\mathcal{C})$-inverse under the Einstein product through operational formulas, with illustrative examples provided for each main result.

\section{Key lemma}

This section presents the main lemmas and definitions utilized in this paper.
Stanimirovi\'c {et al.} defined the range and null space of tensors in \cite{LMA685} as follows.
An appropriate zero tensor is marked by \(\mathcal{O}\).

\begin{definition}{\rm \cite{LMA685}}
The null space and range of \(\mathcal{A} \in \mathbb{C}^{S(s) \times T(t)}\), respectively, are introduced by
\[
\aunclfamily{\bf N}(\mathcal{A}) = \left\{ \mathcal{Z} = (z_{q_1, \ldots, q_q}) \in \mathbb{C}^{T(t)} \mid \mathcal{A} *_t \mathcal{Z} = \mathcal{O} \in \mathbb{C}^{S(s)} \right\} \subseteq \mathbb{C}^{T(t)}.
\]
and
\[
\aunclfamily{\bf R}(\mathcal{A}) = \left\{ \mathcal{A} *_t \mathcal{Z} \mid \mathcal{Z} = (z_{q_1, \ldots, q_q}) \in \mathbb{C}^{T(t)} \right\} \subseteq \mathbb{C}^{S(s)}
\]
\end{definition}

Subsequent definitions for tensors restate the notions of \((\mathcal{B})\)-inverse, \((\mathcal{C})\)-inverse and \( (\mathcal{B}, \mathcal{C}) \)-inverse.

\begin{definition}\cite{LMA685}
For \( \mathcal{B} \in \mathbb{C}^{T(t) \times U(u)} \), \( \mathcal{C} \in \mathbb{C}^{V(v) \times S(s)} \) and \( \mathcal{A} \in \mathbb{C}^{S(s) \times T(t)} \), \( \mathcal{X} \in \mathbb{C}^{T(t) \times S(s)} \) is called:
\begin{itemize}
\item[-] a \((\mathcal{B})\)-inverse of \(\mathcal{A}\) if it satisfies equations \\
$(\text{1TB}) \ \mathcal{X} *_s \mathcal{A} *_t \mathcal{B} = \mathcal{B}  $ and
    $(\text{2TB})\ \mathcal{X} \in \mathcal{B} *_u \mathbb{C}^{U(u) \times T(t)} *_t \mathcal{X};$

\item[-] a \((\mathcal{C})\)-inverse of \(\mathcal{A}\) if it satisfies equations \\
$
    (\text{1TC}) \ \mathcal{C} *_s\mathcal{A} *_t \mathcal{X} =\mathcal{C}$ and
    $(\text{2TC})\ \mathcal{X} \in \mathcal{X} *_s \mathbb{C}^{S(s) \times V(v)} *_v\mathcal{C};
$

\item[-] the \( (\mathcal{B}, \mathcal{C}) \)-inverse of \( \mathcal{A}\) if fulfills
\begin{center}
 $(1T) \quad  \mathcal{C} *_s \mathcal{A} *_t \mathcal{X} = \mathcal{C},\ \mathcal{X} *_s \mathcal{A} *_t \mathcal{B} = \mathcal{B}$,\\
 $(2T) \quad  \mathcal{X} \in \left( \mathcal{X} *_s \mathbb{C}^{S(s) \times V(v)} *_v \mathcal{C} \right) \cap \left( \mathcal{B} *_u \mathbb{C}^{U(u) \times T(t)} *_t \mathcal{X} \right)$.
\end{center}
\end{itemize}
\end{definition}

Lemma \ref{RBinverse} establishes the equivalence between an outer inverse of $\mathcal{A}$ with a prescribed range $\aunclfamily{\bf R}(\mathcal{B})$ (resp. null space $\aunclfamily{\bf N}(\mathcal{C})$) and the corresponding $(\mathcal{B})$-inverse (resp. $(\mathcal{C})$-inverse).
\begin{lemma}\label{RBinverse}{\rm \cite{LMA685}}
Let \(\mathcal{A} \in \mathbb{C}^{S(s) \times T(t)} \), \(\mathcal{X} \in \mathbb{C}^{T(t) \times S(s)} \), \( \mathcal{B} \in \mathbb{C}^{T(t) \times U(u)} \), \( C \in \mathbb{C}^{V(v) \times S(s)} \).
The tensor $\mathcal{A}_{\aunclfamily{\bf R}(\mathcal{B}),*}^{(2)}$ is an outer inverse of $\mathcal{A}$ with predefined range $\aunclfamily{\bf R}(\mathcal{B})$ and $\mathcal{A}_{*,\aunclfamily{\bf N}(\mathcal{C})}^{(2)}$ denotes the outer inverse of $\mathcal{A}$ with prescribed null space $\aunclfamily{\bf N}(\mathcal{C})$.
Then
\[
\mathcal{X} := \mathcal{A}_{\aunclfamily{\bf R}(\mathcal{B}),*}^{(2)} \quad \Longleftrightarrow \quad \mathcal{X} \text{ is } (\mathcal{B})\text{-inverse of }\mathcal{A};
\]
and
\[
\mathcal{X} := \mathcal{A}_{*,\aunclfamily{\bf N}(\mathcal{C})}^{(2)} \quad \Longleftrightarrow \quad \mathcal{X} \text{ is } (C)\text{-inverse of }\mathcal{A}.
\]
\end{lemma}

Lemma \ref{ABLeqB1A1} gives the conditions for the reverse order law as they apply to the inner inverse of tensors.
\begin{lemma}\label{ABLeqB1A1}{\rm \cite{LMA652}}
Let $\mathcal{P}\in\mathbb{C}^{S(s)\times T(t)}$ and $\mathcal{Q}\in\mathbb{C}^{T(t)\times U(u)}$.
Then
$(\mathcal{P}*_t\mathcal{Q})^{(1)}=\mathcal{Q}^{(1)}*_t\mathcal{P}^{(1)}$ holds true if and only if
\begin{equation}\label{0.1}
(\mathcal{P}^{(1)}*_s\mathcal{P}*_t\mathcal{Q}*_u\mathcal{Q}^{(1)})^2=\mathcal{P}^{(1)}*_s\mathcal{P}*_t\mathcal{Q}*_u\mathcal{Q}^{(1)}.
\end{equation}
\end{lemma}

Result given in Lemma \ref{Chai1Inverse2T} extends the perturbation of inner inverses from a Banach space, derived in \cite[Theorem 2.1]{LAA94}, to analogous result in the tensor environment.
More precisely, Lemma \ref{Chai1Inverse2T} investigates an aspect of the {\bf Problem 1.}
The results obtained in Lemma \ref{Chai1Inverse2T} show that a small perturbation $\mathcal{E}$ in $\mathcal{A}$, satisfying $\|\mathcal{A}^{(1)}\|\|\mathcal{E}\| < 1$, causes the perturbation in computing $\mathcal{A}^{(1)}$ defined by the multiplicative factors $\bigdelta =(\mathcal{I} +\mathcal{A}^{(1)}*_s \mathcal{E})^{-1}$ and $\bigrho =(\mathcal{I} + \mathcal{E}*_t \mathcal{A}^{(1)})^{-1}$ under the conditions (i) and (ii) in Lemma \ref{Chai1Inverse2T}.
In the essence, Lemma \ref{Chai1Inverse2T} investigates the relationship $ (\mathcal{A} + \mathcal{E})^{(1)}=\mathcal{A}^{(1)} *_s\bigrho =\bigdelta*_t\mathcal{A}^{(1)}$ under the specified conditions (i) and (ii).

The notations $\bigdelta =(\mathcal{I} +\mathcal{A}^{(1)}*_s \mathcal{E})^{-1}$ and $\bigrho =(\mathcal{I} + \mathcal{E}*_t \mathcal{A}^{(1)})^{-1}$ will be used in further text.

\begin{lemma}\label{Chai1Inverse2T}
 Let \( \mathcal{A} \in \mathbb{C}^{S(s) \times T(t)} \) with an inner inverse \( \mathcal{A}^{(1)} \in \mathbb{C}^{T(t)\times S(s) } \) and \( \mathcal{E} \in \mathbb{C}^{S(s) \times T(t)} \)
be a sufficiently small perturbation such that \( \|\mathcal{A}^{(1)}\|\|\mathcal{E}\| < 1 \).
 Then
\begin{equation}\label{EquHlr}
\mathcal {H} = \mathcal{A}^{(1)}*_s \bigrho = \bigdelta *_t \mathcal{A}^{(1)}
\end{equation}
represents an inner inverse of \( \mathcal{D} = \mathcal{A} + \mathcal{E} \) if and only if the subsequent statements are valid:
\begin{enumerate}
    \item[$(i)$] \( \aunclfamily{\bf R}(\mathcal{D}) \cap \aunclfamily{\bf N}(\mathcal{A}^{(1)}) = \{0\}; \)
    \item[$(ii)$] \( \bigrho *_s \mathcal{D}*_t \aunclfamily{\bf N}(\mathcal{A}) \subseteq \aunclfamily{\bf N}(\mathcal{A}^{(1)}*_s \mathcal{A}*_t \mathcal{A}^{(1)} - \mathcal{A}^{(1)}). \)
\end{enumerate}
\end{lemma}

\begin{proof}
The tensor $\mathcal {H}$ defined by \eqref{EquHlr}
%From  \( \mathcal {H} = \mathcal{A}^{(1)}*_s (\mathcal{I} + \mathcal{E}*_t \mathcal{A}^{(1)})^{-1} = (\mathcal{I} + \mathcal{A}^{(1)}*_s \mathcal{E})^{-1}*_t \mathcal{A}^{(1)} \)  is easy to see that
satisfies \( \aunclfamily{\bf R}(\mathcal {H}) = \aunclfamily{\bf R}(\mathcal{A}^{(1)}) \) and \( \aunclfamily{\bf N}(\mathcal {H}) = \aunclfamily{\bf N}(\mathcal{A}^{(1)}) \).
The notations $\bigAlpha = \mathcal{I} - \mathcal{A}^{(1)}*_s \mathcal{A}$ and $\bigbBeta=\mathcal{A}^{(1)}*_s \mathcal{A} - \mathcal{I}$  will be used in further proof.

\smallskip
\textbf{Necessity.} Let $\mathcal X \in \mathbb{C}^{T(t)}$ and $\mathcal{Y} \in \mathbb{C}^{N(m)}$.
If \( \mathcal {H} \) is an inner inverse of \( \mathcal{D} \), then for any \(  \mathcal{Y} \in \aunclfamily{\bf R}(\mathcal{D}) \cap \aunclfamily{\bf N}(\mathcal{A}^{(1)}) \)
it follows that \(  \mathcal{Y} \in \aunclfamily{\bf R}(\mathcal{D}) \cap \aunclfamily{\bf N}(\mathcal {H})\).
Consequently, there exists an  $\mathcal X  \in \mathbb{C}^{T(t)}$ satisfying \(  \mathcal{Y} = \mathcal{D}*_t \mathcal X \) and \( \mathcal {H}*_s \mathcal{D}*_t \mathcal X = 0 \).
Thus \(  \mathcal{Y} = \mathcal{D}*_t \mathcal X = \mathcal{D}*_t \mathcal {H}*_s \mathcal{D}*_t \mathcal X = 0 \).
This implies \( \aunclfamily{\bf R}(\mathcal{D}) \cap \aunclfamily{\bf N}(\mathcal{A}^{(1)}) = \{0\} \).
Next, for all \(  \mathcal X \in \aunclfamily{\bf N}(\mathcal{A}) \), the identity \( \mathcal {H}*_s \mathcal{D}*_t \mathcal {H}*_s \mathcal{D}*_t \mathcal X = \mathcal {H}*_s \mathcal{D}*_t \mathcal X \)
implies \( (\mathcal {H}*_s \mathcal{D} - \mathcal{I})*_t \mathcal {H}*_s \mathcal{D}*_t \mathcal X = 0 \).
Hence, we conclude
\[
\left(\bigdelta *_t\mathcal{A}^{(1)} *_s \mathcal{D} - \mathcal{I}\right)*_t \mathcal {H}*_s \mathcal{D}*_t \mathcal X = 0
\]
and
\[
\bigdelta*_t \left[\mathcal{A}^{(1)}*_s \mathcal{D} - (\mathcal{I} + \mathcal{A}^{(1)}*_s \mathcal{E})\right]*_t \mathcal {H}*_s \mathcal{D}*_t \mathcal X = 0,
\]
i.e., \( \bigbBeta*_t \mathcal {H}*_s \mathcal{D}*_t \mathcal X = 0 \).
Applying definition of $\mathcal{H}$ in \eqref{EquHlr} it can be concluded
\[
(\mathcal{A}^{(1)}*_s \mathcal{A}*_t \mathcal{A}^{(1)} - \mathcal{A}^{(1)})*_s \bigrho *_s \mathcal{D}*_t \mathcal X = \bigbBeta*_t \mathcal{A}^{(1)}*_s \bigrho*_s \mathcal{D}*_t \mathcal X
\]
\[
= \bigbBeta*_t \mathcal {H}*_s \mathcal{D}*_t \mathcal X = 0.
\]
Therefore, $(ii)$ is verified.
%\( (\mathcal{I} + \mathcal{E}*_t \mathcal{A}^{(1)})^{-1}*_s \mathcal{D}*_t \aunclfamily{\bf N}(\mathcal{A}) \subseteq \aunclfamily{\bf N}(\mathcal{A}^{(1)}*_s \mathcal{A}*_t \mathcal{A}^{(1)} - \mathcal{A}^{(1)}) \).

\smallskip
\textbf{Sufficiency.} For all $\mathcal X \in \mathbb{C}^{T(t)}$,  it follows \(\bigAlpha *_t \mathcal X \in \aunclfamily{\bf N}(\mathcal{A}) \).
%Next, let us write $\mathcal{D}^{(1)'}=\mathcal{A}^{(1)}*_s (\mathcal{I} + \mathcal{E}*_t \mathcal{A}^{(1)})^{-1} = (\mathcal{I} + \mathcal{A}^{(1)}*_s \mathcal{E})^{-1}*_t \mathcal{A}^{(1)}$.
Then by $(ii)$, we can obtain
\begin{align*}
\mathcal {H}*_s &(\mathcal{D}*_t\mathcal {H}*_s \mathcal{D}*_t \mathcal X - \mathcal{D}*_t \mathcal X) =\mathcal {H}*_s \mathcal{D}*_t (\mathcal {H}*_s \mathcal{D} - \mathcal{I})*_t \mathcal X \\
&= \bigdelta *_t \mathcal{A}^{(1)}*_s \mathcal{D}*_t (\mathcal {H}*_s \mathcal{D} - \mathcal{I})*_t \mathcal X \\
&= \bigdelta *_t \mathcal{A}^{(1)}*_s \mathcal{D}*_t \bigdelta *_t [\mathcal{A}^{(1)}*_s \mathcal{D} - (\mathcal{I} + \mathcal{A}^{(1)}*_s \mathcal{E})] *_t \mathcal X \\
&= \bigdelta *_t \mathcal{A}^{(1)}*_s \mathcal{D}*_t \bigdelta *_t \bigbBeta*_t \mathcal X \\
&= \bigdelta *_t \mathcal{A}^{(1)}*_s (\mathcal{A}+\mathcal{E})*_t \bigdelta *_t \bigbBeta*_t \mathcal X \\
&= \bigdelta *_t (\bigbBeta + \mathcal{I} + \mathcal{A}^{(1)}*_s \mathcal{E})*_t \bigdelta *_t \bigbBeta*_t \mathcal X \\
&= \bigdelta *_t \left[\bigbBeta*_t \bigdelta *_t \bigbBeta*_t \mathcal X + \bigbBeta*_t \mathcal X\right] \\
&= \bigdelta *_t \bigbBeta*_t \left(\bigdelta - \mathcal{I}\right)*_t \bigbBeta*_t \mathcal X \\
&= \bigdelta *_t \bigbBeta*_t \bigdelta *_t \left[\mathcal{I} - (\mathcal{I} + \mathcal{A}^{(1)}*_s \mathcal{E})\right]*_t \bigbBeta*_t \mathcal X \\
&= \bigdelta *_t \bigbBeta*_t \bigdelta *_t \mathcal{A}^{(1)}*_s \mathcal{E}*_t \bigAlpha *_t \mathcal X \\
&= \bigdelta *_t \bigbBeta*_t \mathcal{A}^{(1)}*_s \bigrho *_s \mathcal{E}*_t \bigAlpha *_t \mathcal X \\
&= \bigdelta *_t \left(\mathcal{A}^{(1)}*_s \mathcal{A}*_t \mathcal{A}^{(1)} - \mathcal{A}^{(1)}\right)*_s \bigrho *_s \mathcal{D}*_t\bigAlpha *_t \mathcal X \\
&= \bigdelta *_t 0 = 0.
\end{align*}

Thus, \( \mathcal{D}*_t\mathcal {H}*_s \mathcal{D}*_t \mathcal X - \mathcal{D}*_t \mathcal X \in \aunclfamily{\bf N}(\mathcal {H}) \cap \aunclfamily{\bf R}(\mathcal{D}) \).
Therefore by $(i)$ and \( \aunclfamily{\bf N}(\mathcal {H}) = \aunclfamily{\bf N}(\mathcal{A}^{(1)}) \), it can be concluded
\( \mathcal{D}*_t\mathcal {H}*_s \mathcal{D}*_t \mathcal X = \mathcal{D}*_t \mathcal X \), for arbitrary $\mathcal X \in \mathbb{C}^{T(t)}$.
Therefore, $\mathcal {H}$ defined by \eqref{EquHlr}
%\(\mathcal {H}=\mathcal{A}^{(1)}*_s (\mathcal{I} + \mathcal{E}*_t \mathcal{A}^{(1)})^{-1} = (\mathcal{I} + \mathcal{A}^{(1)}*_s \mathcal{E})^{-1}*_t \mathcal{A}^{(1)} \)
is an inner inverse of \( \mathcal{D} \).
\end{proof}

\begin{lemma}\label{BInverseEquation}{\rm \cite{LMA685}}
If $\mathcal{A}\in\mathbb{C}^{S(s)\times T(t)}$and $\mathcal{B}\in\mathbb{C}^{T(t)\times K(k)}$
satisfy $\rra{\mathcal{A}*_t\mathcal{B}}=\rra{\mathcal{B}}$, then
\begin{equation}\label{0.31}
\mathcal{A}^{(2)}_{\aunclfamily{\bf R}(\mathcal{B}),*}=\mathcal{B}*_k(\mathcal{A}*_t\mathcal{B})^{(1)}.
\end{equation}
\end{lemma}

\begin{lemma}\label{CInverseEquation}{\rm \cite{LMA685}}
Let $\mathcal{A}\in\mathbb{C}^{S(s)\times T(t)}$ and $\mathcal{C}\in\mathbb{C}^{L(l)\times S(s)}$.
Under the constraint $\rra{\mathcal{C}*_s\mathcal{A}}=\rra{\mathcal{C}}$, it can be obtained
\begin{equation}\label{0.32}
\mathcal{A}^{(2)}_{*,\aunclfamily{\bf N}(\mathcal{C})}=(\mathcal{C}*_s\mathcal{A})^{(1)}*_l\mathcal{C}.
\end{equation}
\end{lemma}

\begin{lemma}\label{BCInverseEquation}{\rm \cite{LMA685}}
Let $\mathcal{A}\in\mathbb{C}^{S(s)\times T(t)}$, $\mathcal{B}\in\mathbb{C}^{T(t)\times K(k)}$ and $\mathcal{C}\in\mathbb{C}^{L(l)\times S(s)}$. If $\rra{\mathcal{C}*_s\mathcal{A}*_t\mathcal{B}}=\rra{\mathcal{C}}=\rra{\mathcal{B}}$, it follows
\begin{equation}\label{0.33}
\mathcal{A}^{(2)}_{\aunclfamily{\bf R}(\mathcal{B}),\aunclfamily{\bf N}(\mathcal{C})}=\mathcal{B}*_k(\mathcal{C}*_s\mathcal{A}*_t\mathcal{B})^{(1)}*_l\mathcal{C}.
\end{equation}
\end{lemma}
\begin{lemma}{\rm \cite{FTPBF352}}\label{le-inv}
Let $\mathcal{F}\in\mathbb{C}^{S(s)\times S(s)}$ satisfy $\Vert \mathcal{F} \Vert <1$.
In this case, $\mathcal{I}+\mathcal{F}$ is invertible and
\begin{equation}\label{0.4}
\Vert (\mathcal{I}+\mathcal{F})^{-1}\Vert \leq \frac{1}{1-\Vert \mathcal{F}\Vert}.
 \end{equation}
\end{lemma}

\section{Perturbation of inner inverse}

In this section, we give the equivalent condition under which \( \mathcal{H} = \mathcal{A}^{(1)}*_s (\mathcal{I} + \mathcal{E}*_t \mathcal{A}^{(1)})^{-1} =\mathcal{A}^{(1)}*_s \bigrho\) serves as an inner inverse of
\( \mathcal{D} = \mathcal{A} + \mathcal{E} \).
Additionally, we provide the expression for the inner inverse of the tensor perturbation.
Perturbation of such type is examined in \cite[Theorem 2.2]{LAA94} for Banach spaces.

\begin{lemma}\label{Chai1Inverse5T}
Assume that \( \mathcal{A} \in \mathbb{C}^{S(s) \times T(t)} \) and \( \mathcal{D} \in \mathbb{C}^{S(s) \times T(t)} \) with an inner inverse
\( \mathcal {H} \in \mathbb{C}^{ T(t) \times S(s)}\in \mathcal{D}\{1\} \).
Let \(\mathcal{E} \in  \mathbb{C}^{S(s) \times T(t)}  \) satisfy \( \|\mathcal{A}^{(1)}\|\|\mathcal{E}\| < 1 \).
The next statements are equivalent:
\begin{enumerate}
    \item[$(i)$] \( \mathcal {H} = \mathcal{A}^{(1)}*_s \bigrho \) is an inner inverse of \( \mathcal{D} = \mathcal{A} + \mathcal{E} \);
    \item[$(ii)$] \( \aunclfamily{\bf N}(\mathcal{A}*_t \mathcal{A}^{(1)}) \cap  \aunclfamily{\bf R}(\mathcal{D})= \{0\} \);
    \item[$(iii)$] \( \bigdelta *_t \aunclfamily{\bf N}(\mathcal{A})  = \aunclfamily{\bf N}(\mathcal{D}) \).
\end{enumerate}
\end{lemma}
\begin{proof}
The notations $\bigAlpha = \mathcal{I} - \mathcal{A}^{(1)}*_s \mathcal{A}$ and $\bigbBeta=\mathcal{A}^{(1)}*_s \mathcal{A} - \mathcal{I}$  will be used in further proof.

$(i)$\(\Rightarrow\) $(ii)$. The condition \( \aunclfamily{\bf R}(\mathcal{D}) \cap \aunclfamily{\bf N}(\mathcal {H}*_s \mathcal{D}) = \{0\} \) is satisfied for arbitrary inner inverse \( \mathcal {H} \) of \( \mathcal{D}\).
For any \(  \mathcal{Y} \in \aunclfamily{\bf N}(\mathcal{A}*_t \mathcal{A}^{(1)}) \), it follows \( \mathcal{A}^{(1)}*_s \mathcal{Y} \in \aunclfamily{\bf N}(\mathcal{A}) \) and
\(  \mathcal{Y} + \mathcal{E}*_t \mathcal{A}^{(1)}*_s \mathcal{Y} =  \mathcal{Y} +\mathcal{D}*_t \mathcal{A}^{(1)}*_s \mathcal{Y} \).
Hence,
\[
 \mathcal{Y} = \bigrho*_s \mathcal{Y} + \bigrho *_s \mathcal{D}*_t \mathcal{A}^{(1)}*_s \mathcal{Y}.
\]
Thus, by Lemma \ref{Chai1Inverse2T}, we obtain
\[
(\mathcal{A}^{(1)}*_s \mathcal{A}*_t \mathcal{A}^{(1)} - \mathcal{A}^{(1)})*_s \mathcal{Y} = (\mathcal{A}^{(1)}*_s \mathcal{A}*_t \mathcal{A}^{(1)} - \mathcal{A}^{(1)})*_s \bigrho *_s \mathcal{Y},
\]
i.e., \( \mathcal{A}^{(1)}*_s \mathcal{Y} = (\mathcal{I} - \mathcal{A}^{(1)}*_s \mathcal{A})*_t \mathcal {H}*_s \mathcal{Y} \).
Therefore,
\begin{align*}
\mathcal{A}^{(1)}*_s \mathcal{D}*_t \mathcal {H}*_s \mathcal{Y} &= \mathcal{A}^{(1)}*_s (\mathcal{A} + \mathcal{E})*_t \mathcal {H}*_s \mathcal{Y} \\
&= (\bigbBeta + \mathcal{I} + \mathcal{A}^{(1)}*_s \mathcal{E})*_t \mathcal {H}*_s \mathcal{Y} \\
&= \bigbBeta*_t \mathcal {H}*_s \mathcal{Y} + \mathcal{A}^{(1)}*_s \mathcal{Y} = 0
\end{align*}
and so \( \mathcal {H}*_s \mathcal{D}*_t \mathcal {H}*_s \mathcal{Y} = 0 \).
Hence, \( \mathcal{D}*_t \mathcal {H}*_s \mathcal{Y} = \mathcal{D}*_t \mathcal {H}*_s \mathcal{D}*_t \mathcal {H}*_s \mathcal{Y} = 0 \), i.e.,
\(  \mathcal{Y} \in \aunclfamily{\bf N}(\mathcal{D}*_t \mathcal {H}) \).
This implies \( \aunclfamily{\bf N}(\mathcal{A}*_t \mathcal{A}^{(1)}) \subseteq \aunclfamily{\bf N}(\mathcal{D}*_t \mathcal {H}) \).
Utilizing \(\aunclfamily{\bf R}(\mathcal{D}) \cap \aunclfamily{\bf N}(\mathcal{D}*_t \mathcal {H}) = \{0\} \), we obtain
\( \aunclfamily{\bf N}(\mathcal{A}*_t \mathcal{A}^{(1)}) \cap\aunclfamily{\bf R}(\mathcal{D})= \{0\} \).

$(ii)$ \(\Rightarrow\) $(i)$. For all \(\mathcal{X}\in \mathbb{C}^{ T(t)} \), it can be concluded
\begin{align*}
\mathcal{A}*_t \mathcal{A}^{(1)}*_s (\mathcal{D}*_t \mathcal {H}*_s \mathcal{D}*_t \mathcal X - \mathcal{D}*_t \mathcal X) &= \mathcal{A}*_t \mathcal{A}^{(1)}*_s (\mathcal{A} + \mathcal{E})*_t (\mathcal {H}*_s \mathcal{D} - \mathcal{I})*_t \mathcal X \\
% &= \mathcal{A}*_t \left(\bigbBeta + \mathcal{I} + \mathcal{A}^{(1)}*_s \mathcal{E}\right)*_t \bigdelta  *_t \left(\mathcal{A}^{(1)}*_s \mathcal{D} - \mathcal{I} - \mathcal{A}^{(1)}*_s \mathcal{E}\right)*_t \mathcal X \\
&= \mathcal{A}*_t (\mathcal{I} + \mathcal{A}^{(1)}*_s \mathcal{E})*_t\bigdelta*_t\bigbBeta*_t \mathcal X = 0,
\end{align*}
which further leads to
\( \mathcal{D}*_t \mathcal {H}*_s \mathcal{D}*_t \mathcal X - \mathcal{D}*_t \mathcal X \in \aunclfamily{\bf N}(\mathcal{A}*_t \mathcal {H}) \).
Combining the last identity with \( \mathcal{D}*_t \mathcal {H}*_s \mathcal{D}*_t \mathcal X - \mathcal{D}*_t \mathcal X \in \aunclfamily{\bf R}(\mathcal{D}) \),
by $(ii)$, it can be obtained
\( \mathcal{D}*_t \mathcal {H}*_s \mathcal{D}*_t \mathcal X = \mathcal{D}*_t \mathcal X \), i.e., \( \mathcal {H} \) is an inner inverse of \( \mathcal{D} \).

\smallskip
\noindent $(iii)$ \(\Rightarrow\) $(ii)$.
The assumption \(  \mathcal{Y} \in \aunclfamily{\bf R}(\mathcal{D}) \cap \aunclfamily{\bf N}(\mathcal{A}*_t \mathcal{A}^{(1)}) \) ensures the existence of
\(\mathcal{X}\in \mathbb{C}^{ T(t)} \) with the properties \(  \mathcal{Y} = \mathcal{D}*_t \mathcal X \) and \( \mathcal{A}*_t \mathcal{A}^{(1)}*_s \mathcal{D}*_t \mathcal X = 0 \).
Hence,
\begin{align*}
\mathcal{A}*_t (\mathcal{I} + \mathcal{A}^{(1)}*_s \mathcal{E})*_t \mathcal X &= \mathcal{A}*_t \mathcal X + \mathcal{A}*_t \mathcal{A}^{(1)}*_s \mathcal{E}*_t \mathcal X \\
&= \mathcal{A}*_t \mathcal{A}^{(1)}*_s \mathcal{D}*_t \mathcal X = 0.\end{align*}

This means \( (\mathcal{I} + \mathcal{A}^{(1)}*_s \mathcal{E})*_t \mathcal X \in \aunclfamily{\bf N}(\mathcal{A}) \).
By $(iii)$, \(  \mathcal X \in \aunclfamily{\bf N}(\mathcal{D}) \), and so \(  \mathcal{Y} = \mathcal{D}*_t \mathcal X = 0 \).
Now we get \( \aunclfamily{\bf R}(\mathcal{D}) \cap \aunclfamily{\bf N}(\mathcal{A}*_t \mathcal{A}^{(1)}) = \{0\} \).

\smallskip
\noindent $(ii)$ \(\Rightarrow\) $(iii)$.\ \ On the one hand, let $\mathcal{Y} \in \aunclfamily{\bf N} (\mathcal{D})$ and  $\mathcal X=(\mathcal{I}+\mathcal{A}^{(1)}*_s \mathcal{E})*_t \mathcal{Y}$.
Hence, $\mathcal{D}*_t \mathcal{Y}=(\mathcal{A}+\mathcal{E})*_t \mathcal{Y}=0$,\ \ $\mathcal{A}*_t \mathcal{Y}=-\mathcal{E}*_t \mathcal{Y}$. Then
\begin{align*}
\mathcal{A}*_t \mathcal X
&=\mathcal{A}*_t \mathcal{Y}+\mathcal{A}*_t \mathcal{A}^{(1)}*_s \mathcal{E}*_t \mathcal{Y}\\
&=(\mathcal{I}-\mathcal{A}*_t \mathcal{A}^{(1)})*_s \mathcal{A}*_t \mathcal{Y}=0.
\end{align*}
So, $\mathcal X\in \aunclfamily{\bf N}(\mathcal{A})$, and then $\mathcal{Y}=\bigdelta*_t \mathcal X$ and $\bigdelta *_t \aunclfamily{\bf N}(\mathcal{A})\supseteq \aunclfamily{\bf N}(\mathcal{D}).$

\smallskip
On the other, the assumption $\mathcal X \in \aunclfamily{\bf N}(\mathcal{A})$ and $\mathcal{Y}=\bigdelta *_t \mathcal X$ leads to $(\mathcal{I}+\mathcal{A}^{(1)}*_s \mathcal{E})*_t \mathcal{Y}=\mathcal X$, $\mathcal{Y}=\mathcal X-\mathcal{A}^{(1)}*_s \mathcal{E}*_t \mathcal{Y}$, which implies
\[
\mathcal{A}*_t \mathcal{Y}=\mathcal{A}*_t \mathcal X-\mathcal{A}*_t \mathcal{A}^{(1)}*_s \mathcal{E}*_t \mathcal{Y}=-\mathcal{A}*_t \mathcal{A}^{(1)}*_s \mathcal{E}*_t \mathcal{Y}.
\]
Consider
$\mathcal{D}*_t \mathcal{Y}=(\mathcal{A}+\mathcal{E})*_t \mathcal{Y}=\mathcal{E}*_t \mathcal{Y}-\mathcal{A}*_t \mathcal{A}^{(1)}*_s \mathcal{E}*_t \mathcal{Y}=(\mathcal{I}-\mathcal{A}*_t \mathcal{A}^{(1)})*_s \mathcal{E}*_t \mathcal{Y}$,
i.e., $\mathcal{D}*_t \mathcal{Y}\in \aunclfamily{\bf N}(\mathcal{A}*_t \mathcal{A}^{(1)})$.
In this case, based on $(ii)$, it can be concluded that  $\mathcal{D}*_t \mathcal{Y}=0$.
Accordingly, $\mathcal{Y} \in \aunclfamily{\bf N}(\mathcal{D})$, and further
$\bigdelta *_t \aunclfamily{\bf N}(\mathcal{A})\subseteq \aunclfamily{\bf N}(\mathcal{D})$,
which means \( \bigdelta *_t \aunclfamily{\bf N}(\mathcal{A}) = \aunclfamily{\bf N}(\mathcal{D}) \).
\end{proof}

The following theorem is studied in the context of Banach spaces in \cite[Corollary 2.2]{LAA94}.
We have especially renewed our approach to prove this result using a new method.

\begin{theorem}\label{Chai1Inverse1Trank}(Finite Rshrank Theorem)
Let \( \mathcal{A} \in \mathbb{C}^{S(s) \times T(t)} \) be a tensor with an inner inverse \( \mathcal{A}^{(1)} \in \mathbb{C}^{T(t) \times S(s)} \) satisfy \(\rra{\mathcal{A}} < +\infty\).
If \(\mathcal{E} \in \mathbb{C}^{S(s) \times T(t)}\) satisfies the condition \(\|\mathcal{A}^{(1)}\|\|\mathcal{E}\| < 1\), then
 \( \mathcal {H} = \mathcal{A}^{(1)}*_s \bigrho \) is an inner inverse of \( \mathcal{D} = \mathcal{A} + \mathcal{E} \) if and only if
\[
\rra{\mathcal{D}} =\rra{\mathcal{A}} < +\infty.
\]
\end{theorem}

\begin{proof}
The definition $\mathfrak{N}=\prod_{i=1}^n N_i$ and the definition $\mathrm{RSH}\left(\aunclfamily{\bf N}(A)\right)=\left\{\mathrm{rsh}(Q)\mid Q\in \aunclfamily{\bf N}(A)\right\}$ of the RSH operator from \cite{LMA685} will be useful.

Necessity follows from the part $(iii)$ of Lemma \ref{Chai1Inverse5T}.
Indeed, since $\bigrho$ is invertible, its effect does not change the dimension of the space.
Therefore, it follows
$\mathrm{dim}(\aunclfamily{\bf N}(\mathcal{D}))=\mathrm{dim}(\aunclfamily{\bf N}(\mathcal{A}))=\mathrm{dim}(RSH(\aunclfamily{\bf N}(A)))$.
After that, it follows $\mathrm{dim}(\aunclfamily{\bf N}(\mathcal{A}))=\mathrm{dim}(\aunclfamily{\bf N}(A))=\mathfrak{N}-\mathrm{dim}(\aunclfamily{\bf R}(\mathcal{A}))$.
As a result, we conclude $\mathrm{dim}(\aunclfamily{\bf R}(\mathcal{D}))=\mathrm{dim}(\aunclfamily{\bf R}(\mathcal{A}))$, or equivalently $\rra{\mathcal{D}}=\rra{\mathcal{A}}$.

% $\mathcal{A}$   and $\mathcal{B}$ have the same size, then  $rshrank(\mathcal{B})=rshrank(\mathcal{A})$.

% Assume there exists a nonzero  \( \mathcal{Y} \in \aunclfamily{\bf R}(\mathcal{B}) \cap \aunclfamily{\bf N}(\mathcal{A}*_t \mathcal{A}^{(1)}) \). Then there exists \( \mathcal X \) such that \( \mathcal{Y} = \mathcal{B}*_s \mathcal X = (\mathcal{A} + \mathcal{E})*_t \mathcal X \), and \( \mathcal{A}*_t \mathcal{A}^{(1)}*_s \mathcal{Y} = 0 \).

% So we obtain:
% \[
% \mathcal{A}*_t \mathcal{A}^{(1)}*_s \mathcal{Y} =\mathcal{A}*_t \mathcal{A}^{(1)}*_s (\mathcal{A} + \mathcal{E})*_t \mathcal X = \mathcal{A}*_s \mathcal X + \mathcal{A}*_t \mathcal{A}^{(1)}*_s \mathcal{E}*_s \mathcal X = 0.
% \]

% Then, \(\mathcal{A}*_s \mathcal X = -\mathcal{A}*_t \mathcal{A}^{(1)}*_s \mathcal{E}*_s \mathcal X \) indicates that \( \mathcal{A}*_s \mathcal X \in \aunclfamily{\bf R}(\mathcal{A}) \), then  \( \mathcal{A}*_t \mathcal{A}^{(1)}*_s \mathcal{E}*_s \mathcal X \) also belongs to \( \aunclfamily{\bf R}(\mathcal{A}) \), i.e., $\aunclfamily{\bf R}(\mathcal{A}*_t\mathcal{A}^{(1)})\subseteq \aunclfamily{\bf R}(\mathcal{A})$.

Next, we shall show the sufficiency.
Given that \(\rra{\mathcal{A}}=\rra{\mathcal{D}}\), it follows $\mathrm{dim}(\aunclfamily{\bf R}(\mathcal{A}))=\mathrm{dim}(\aunclfamily{\bf R}(\mathcal{D}))$.
Since $\mathcal{A}*_t \mathcal{A}^{(1)}$ is a projection onto $\aunclfamily{\bf R} (\mathcal{A})$, it follows that $\aunclfamily{\bf R}(\mathcal{A})=\aunclfamily{\bf R}(\mathcal{A}*_t\mathcal{A}^{(1)})$ and
$\aunclfamily{\bf N} (\mathcal{A}*_t \mathcal{A}^{(1)})$ is the complement space of $\aunclfamily{\bf R} (\mathcal{A})$.
Assume there exists a nonzero element \( \mathcal{Y} \in \aunclfamily{\bf R}(\mathcal{D}) \cap \aunclfamily{\bf N}(\mathcal{A}*_t \mathcal{A}^{(1)}) \).
This means that \( \aunclfamily{\bf R}(\mathcal{D}) \) contains a nonzero element intersecting with \( \aunclfamily{\bf N}(\mathcal{A}*_t \mathcal{A}^{(1)}) \).
Consequently, we would have $\rra{\mathcal{D}}=\mathrm{dim}(\aunclfamily{\bf R}(\mathcal{D}))\geq \mathrm{dim}(\aunclfamily{\bf R}(\mathcal{A}))=\rra{\mathcal{A}}$, which contradicts the assumed reshape rank equality.
Therefore, the only possibility is $\mathcal{Y}=0$.
Hence, $\aunclfamily{\bf R}(\mathcal{D})\cap\aunclfamily{\bf N}(\mathcal{A}*_t \mathcal{A}^{(1)}) = \{0\} $.
According to Lemma \ref{Chai1Inverse5T}, the tensor \( \mathcal {H} = \mathcal{A}^{(1)}*_s \bigrho \) is an inner inverse of \( \mathcal{D} = \mathcal{A} + \mathcal{E} \).
\end{proof}

Algorithm \ref{Alg1} defines an effective procedure to compute $(\mathcal{A} + \mathcal{E})^{(1)}$ by Theorem \ref{Chai1Inverse1Trank}.

\begin{algorithm}[H] \caption{Computation of $(\mathcal{A} + \mathcal{E})^{(1)}$ based on Theorem \ref{Chai1Inverse1Trank}} \label{Alg1}
\begin{algorithmic}[1]
\smallskip
\REQUIRE  $\mathcal{A} \in \mathbb{C}^{S(s) \times T(t)}$ and  $\mathcal{E} \in \mathbb{C}^{S(s) \times T(t)}$

\STATE Compute $ \Vert \mathcal{A}^{(1)}\Vert$ and $\Vert \mathcal{E}\Vert$
\STATE  Compute $\rra{\mathcal{D}}$ and $\rra{\mathcal{A}}$ using Algorithm 3 from \cite{LMA685}
\IF { $ \Vert \mathcal{A}^{(1)}\Vert\Vert \mathcal{E}\Vert<1$ and $\rra{\mathcal{D}}= \rra{\mathcal{A}}$}
\STATE Compute $\mathcal{A}^{(1)}$ using Algorithm 4 from \cite{LMA685}
\STATE  Compute  $\mathcal{A}^{(1)}*_s\bigrho$
\RETURN $(\mathcal{A} + \mathcal{E})^{(1)} =\mathcal{A}^{(1)}*_s\bigrho$
\ENDIF
\end{algorithmic}
\end{algorithm}

\begin{example}\label{expAplusE1inverse}
Assume  $\mathcal{A}\in\mathbb{R}^{(2\times 2)\times (2\times 2)}$, with entries
\begin{center}
$\mathcal{A}(:,:,1,1) = \begin{bmatrix} \ 1 & 1 \\ 1 & 1\ \end{bmatrix}; \quad \mathcal{A}(:,:,1,2) = \begin{bmatrix} \ 0 & 0 \\ 0 & -1\ \end{bmatrix}$\\
$\mathcal{A}(:,:,2,1) = \begin{bmatrix} \ 0 & 0 \\ 0 & -1\ \end{bmatrix}; \quad \mathcal{A}(:,:,2,2) = \begin{bmatrix} \ 0 & 0 \\ 0 & -1\ \end{bmatrix}$.

\end{center}
Consider perturbation tensor $\mathcal{E}_1\in\mathbb{R}^{(2\times 2)\times (2\times 2)}$ with
\begin{center}
$\mathcal{E}_1(:,:,1,1) = \begin{bmatrix} \ 0 & 0.1 \\ 0 & -0.5\ \end{bmatrix}; \quad \mathcal{E}_1(:,:,1,2) = \begin{bmatrix} \ 0 & 0 \\ 0 & 0\ \end{bmatrix}$\\
$\mathcal{E}_1(:,:,2,1) = \begin{bmatrix} \ 0 & 0.5 \\ 0 & 0\ \end{bmatrix}; \quad \mathcal{E}_1(:,:,2,2) = \begin{bmatrix} \ 0 & 0 \\ 0 & 0.5\ \end{bmatrix}$.

\end{center}
Next, applying  Algorithm 4 from \cite{LMA685}, we calculate $\mathcal{A}^{(1)}$:
\begin{center}
$\mathcal{A}^{(1)}(:,:,1,1) = \begin{bmatrix} \ 0.3333 & 0.1111 \\ 0.1111 & 0.1111\ \end{bmatrix}; \quad \mathcal{A}^{(1)}(:,:,1,2) = \begin{bmatrix} \ 0.3333 & 0.1111 \\ 0.1111 & 0.1111\ \end{bmatrix}$\\
$\mathcal{A}^{(1)}(:,:,2,1) = \begin{bmatrix} \ 0.3333 & 0.1111 \\ 0.1111 & 0.1111\ \end{bmatrix}; \quad \mathcal{A}^{(1)}(:,:,2,2) = \begin{bmatrix} \ 0 & -0.3333 \\ -0.3333 & -0.3333\ \end{bmatrix}$.
\end{center}
Under the F-norm, the norms of $\mathcal{A}^{(1)}$ and $\mathcal{E}_1$ can be calculated respectively, and the results are
\begin{center}
$\Vert \mathcal{A}^{(1)} \Vert_F=0.8819$, \ \ \ $\Vert \mathcal{E}_1 \Vert_F=0.8718$, \ \ \ $\Vert \mathcal{A}^{(1)}\Vert_F\Vert \mathcal{E}_1 \Vert_F=0.7688<1$.
\end{center}
Then $\rra{\mathcal{A}}=2,\ \  \rra{\mathcal{A}+\mathcal{E}_1}=3$, and calculation based on  Algorithm \ref{Alg1} gives
\begin{center}$\mathcal{D}^{(1)}(:,:,1,1) = \begin{bmatrix} \ 0.3125 & 0.0625 \\ 0.0625 & 0.0625\ \end{bmatrix}; \quad \mathcal{D}^{(1)}(:,:,1,2) = \begin{bmatrix} \ 0.3125 & 0.0625 \\ 0.0625 & 0.0625\ \end{bmatrix}$\\
$\mathcal{D}^{(1)}(:,:,2,1) = \begin{bmatrix} \ 0.3125 & 0.0625 \\ 0.0625 & 0.0625\ \end{bmatrix}; \quad \mathcal{D}^{(1)}(:,:,2,2) = \begin{bmatrix} \ 0.0625 & -0.3875 \\ -0.3875 & -0.3875\ \end{bmatrix}$.
 \end{center}
However $(\mathcal{A}+\mathcal{E}_1)*_2\mathcal{D}^{(1)}*_2(\mathcal{A}+\mathcal{E}_1) \neq \mathcal{A}+\mathcal{E}_1$, thus $\mathcal{D}^{(1)}=\mathcal{A}^{(1)}*_2(\mathcal{I} + \mathcal{E}_1*_2\mathcal{A}^{(1)})^{-1}$ is not the inner inverse of $\mathcal{A}+\mathcal{E}_1$.

Next consider perturbation tensor $\mathcal{E}_2\in\mathbb{R}^{(2\times 2)\times (2\times 2)}$ with entries
\begin{center}
$\mathcal{E}_2(:,:,1,1) = \begin{bmatrix} \ 0 & 0 \\ 0 & -0.5\ \end{bmatrix}; \quad \mathcal{E}_2(:,:,1,2) = \begin{bmatrix} \ 0 & 0 \\ 0 & 0\ \end{bmatrix};$
$\mathcal{E}_2(:,:,2,1) = \begin{bmatrix} \ 0 & 0 \\ 0 & 0\ \end{bmatrix}; \quad \mathcal{E}_2(:,:,2,2) = \begin{bmatrix} \ 0 & 0 \\ 0 & 0.5\ \end{bmatrix}$.
\end{center}
Further calculion gives
\begin{center}
$\Vert \mathcal{E}_2 \Vert_F=0.7071$,\ \ \
$\Vert \mathcal{A}^{(1)}\Vert_F\Vert \mathcal{E}_2 \Vert_F=0.6236<1$.
\end{center}
Then  $\rra{\mathcal{A}}=2,\ \  \rra{\mathcal{A}+\mathcal{E}_2}=2$, meet the conditions of  Algorithm \ref{Alg1}, we can note
\begin{center}$(\mathcal{A} +\mathcal{E}_2)^{(1)}(:,:,1,1) = \begin{bmatrix} \ 0.3333 & 0.0667 \\ 0.0667 & 0.0667\ \end{bmatrix}; \quad (\mathcal{A} +\mathcal{E}_2)^{(1)}(:,:,1,2) = \begin{bmatrix} \ 0.3333 & 0.0667 \\ 0.0667 & 0.0667\ \end{bmatrix}$;\\
$(\mathcal{A} +\mathcal{E}_2)^{(1)}(:,:,2,1) = \begin{bmatrix} \ 0.3333 & 0.0667 \\ 0.0667 & 0.0667\ \end{bmatrix}; \quad (\mathcal{A} +\mathcal{E}_2)^{(1)}(:,:,2,2) = \begin{bmatrix} \ 0 & -0.4 \\ -0.4 & -0.4\ \end{bmatrix}$.
\end{center}
It can be verified $(\mathcal{A}+\mathcal{E}_2)*_2 (\mathcal{A} +\mathcal{E}_2)^{(1)}*_2(\mathcal{A}+\mathcal{E}_2) = \mathcal{A}+\mathcal{E}_2$.
\end{example}
Next, Theorem \ref{APlusB1} gives the perturbation result for the inner inverse.
\begin{theorem}\label{APlusB1}
Let $\mathcal{A}, \mathcal{E}\in\mathbb{C}^{S(s)\times T(t)}$, $\mathcal{A}^{(1)}\in\mathbb{C}^{T(t)\times S(s) }$ be the inner inverse of $\mathcal{A}$, and consider $ \mathcal{D} = \mathcal{A}+ \mathcal{E}$.
If  \(\|\mathcal{A}^{(1)}\|\|\mathcal{E}\| < 1\) and $\rra{\mathcal{A}}= \rra{\mathcal{D}}$, then
$$  \mathcal{D}^{(1)} - \mathcal{A}^{(1)} =\mathcal{A}^{(1)}*_s (\bigrho- \mathcal{I}),$$
  $$ \frac{\Vert\mathcal{D}^{(1)} - \mathcal{A}^{(1)}\Vert}{\Vert\mathcal{A}^{(1)}\Vert} \leq \frac{2 - \Vert\mathcal{E} *_t \mathcal{A}^{(1)}\Vert}{1 - \Vert\mathcal{E} *_t \mathcal{A}^{(1)} \Vert}.$$

\end{theorem}

\begin{proof}
Utilizing the given condition \( \|\mathcal{A}^{(1)}\| \|\mathcal{E}\| < 1 \), it can be obtained
\[
\|\mathcal{E} *_t \mathcal{A}^{(1)}\| \leq \|\mathcal{A}^{(1)}\| \|\mathcal{E}\| < 1.
\]
By Corollary \ref{Chai1Inverse1Trank}, the inner inverse of \( \mathcal{D} = \mathcal{A} + \mathcal{E} \) is
\[
\mathcal{D}^{(1)} = \mathcal{A}^{(1)} *_s \bigrho,
\]
which implies
\[
   \mathcal{D}^{(1)} - \mathcal{A}^{(1)} =\mathcal{A}^{(1)}*_s (\bigrho- 1).
\]
From \eqref{0.4}, we estimate
\[
\|\bigrho\| \leq \frac{1}{1 - \|\mathcal{E} *_t \mathcal{A}^{(1)}\|},
\]
which initiates
\[
\|\mathcal{D}^{(1)}\| \leq \|\mathcal{A}^{(1)}\| \cdot \frac{1}{1 - \|\mathcal{E} *_t \mathcal{A}^{(1)}\|}.
\]
The application of the triangle inequality results in
\[
\|\mathcal{D}^{(1)} - \mathcal{A}^{(1)}\| \leq \|\mathcal{D}^{(1)}\| + \|\mathcal{A}^{(1)}\| \leq \|\mathcal{A}^{(1)}\| \left( \frac{1}{1 - \|\mathcal{E} *_t \mathcal{A}^{(1)}\|} + 1 \right).
\]
By simplifying the expression on the right-hand side, we obtain the following result:
\[
\frac{\|\mathcal{D}^{(1)} - \mathcal{A}^{(1)}\|}{\|\mathcal{A}^{(1)}\|} \leq \frac{1}{1 - \|\mathcal{E} *_t \mathcal{A}^{(1)}\|} + 1 = \frac{2 - \|\mathcal{E} *_t \mathcal{A}^{(1)}\|}{1 - \|\mathcal{E} *_t \mathcal{A}^{(1)}\|},
\]
which is just desired inequality.
\end{proof}

\begin{example}
Denote $\mathcal{A}$ and $\mathcal{E}_2$ from Example \ref{expAplusE1inverse},  then $\mathcal{D} = \mathcal{A}+\mathcal{E}_2$ and $\mathcal{D}^{(1)}$ is the inner inverse of $\mathcal{D}$.  $\mathcal{D}^{(1)}=(\mathcal{A} +\mathcal{E}_2)^{(1)}$ and $\mathcal{A}^{(1)}$ have been calculated in the Example \ref{expAplusE1inverse}.
Thus,
   $$  \|\mathcal{D}^{(1)} - \mathcal{A}^{(1)}\|_F =0.1764,\ \ \ \Vert\mathcal{E}_2 *_{2} \mathcal{A}^{(1)}\Vert_F=0.2546,$$
which is transformed into
$$\frac{\|\mathcal{D}^{(1)} - \mathcal{A}^{(1)}\|_F}{\|\mathcal{A}^{(1)}\|_F}=0.2000,\ \ \ \frac{2 - \Vert\mathcal{E}_2 *_{2} \mathcal{A}^{(1)}\Vert_F}{1 - \Vert\mathcal{E}_2 *_{2} \mathcal{A}^{(1)} \Vert_F}=2.3415.$$
Therefore, the inequality of Theorem \ref{APlusB1} holds.
\end{example}

\begin{theorem}
Assume that \( \mathcal{A} \in \mathbb{C}^{S(s) \times T(t)} \) with inner inverse $\mathcal{A}^{(1)}$ and $\mathcal{D}=\mathcal{A}+\mathcal{E}$. If $\Vert \mathcal{E}*_t\mathcal{A}^{(1)}\Vert<1 $ and $\mathcal{E}=\mathcal{E}*_t\mathcal{A}^{(1)}*_s\mathcal{A}$ (or $\mathcal{E}=\mathcal{A} *_t\mathcal{A}^{(1)}*_s\mathcal{E}$), then
$$\mathcal{D}^{(1)}=\mathcal{A}^{(1)}*_s\bigrho    =\bigdelta*_t\mathcal{A}^{(1)}.$$
\end{theorem}

\begin{proof} Based on the assumption $\Vert \mathcal{E}*_t\mathcal{A}^{(1)}\Vert<1$, we conclude the existence of $\bigrho$ and $\bigdelta$.
From $\mathcal{D}=(\mathcal{I}+\mathcal{E} *_t\mathcal{A}^{(1)})*_s\mathcal{A}$, we derive
$\mathcal{A}^{(1)}*_s \bigrho*_s\mathcal{D}=\mathcal{A}^{(1)}*_s\mathcal{A}$ and
\begin{equation*}
\aligned
\mathcal{D} *_t \mathcal{A}^{(1)}*_s\bigrho *_s \mathcal{D} &=\mathcal{D} *_t \mathcal{A}^{(1)}*_s\mathcal{A} \\
&=\mathcal{A} *_t \mathcal{A}^{(1)}*_s\mathcal{A} +\mathcal{E} *_t \mathcal{A}^{(1)}*_s\mathcal{A}\\
&=\mathcal{D}.
\endaligned
\end{equation*}
In this way, $\mathcal{A}^{(1)}*_s\bigrho$ is an inner inverse of $\mathcal{D}$.
Now, using
\begin{equation*}
\aligned
(\mathcal{I}+\mathcal{A}^{(1)} *_s\mathcal{E}) *_t \mathcal{A}^{(1)} &=\mathcal{A}^{(1)}+\mathcal{A}^{(1)}*_s \mathcal{E} *_t\mathcal{A}^{(1)}\\
&=\mathcal{A}^{(1)} *_s (\mathcal{I}+\mathcal{E} *_t \mathcal{A}^{(1)}),
\endaligned
\end{equation*}
the equality $\mathcal{A}^{(1)} *_s\bigrho   =\bigdelta *_t\mathcal{A}^{(1)}$ is concluded.
\end{proof}

\section{Perturbation of outer and $(\mathcal{B},\mathcal{C})$-inverses}

Next, we will express perturbations for the $(\mathcal{B})$-inverse, $(\mathcal{C})$-inverse and the $(\mathcal{B},\mathcal{C})$-inverse.

We begin with perturbation formulae for tensor outer inverses.
Theorem \ref{0.3B} addresses {\bf Problem 2.}
This problem involves a perturbation $\mathcal{E}$ in $\mathcal{A}$ that is small enough such that $\|\mathcal{E}*_t\mathcal{A}^{(2)}\|< 1$ holds true.
The research question focuses on understanding how $\mathcal{E}$ affects the computation of $\mathcal{A}^{(2)}$ under specific conditions.
More specifically, the research task is to explore relations between $ (\mathcal{A} + \mathcal{E})^{(2)}$ and $\mathcal{A}^{(2)}$, defined by the multiplicative factors
$\bigDdelta   =(\mathcal{I} + \mathcal{A}^{(2)}*_s\mathcal{E})^{-1}$ and $\bigDrho  =(\mathcal{I} + \mathcal{E}*_t \mathcal{A}^{(2)})^{-1}$.
Essentially, Theorem \ref{0.3B} investigates the relationship of the form $ (\mathcal{A} + \mathcal{E})^{(2)}=\mathcal{A}^{(2)} *_s\bigDrho =\bigDdelta   *_t\mathcal{A}^{(2)}$ under certain specified conditions.
Subsequent results in Theorem \ref{0.3B} will further explore the issue concerning outer inverses with specified range or/and null space.

\begin{theorem}\label{0.3B}
    Let \( \mathcal{A} \in \mathbb{C}^{S(s) \times T(t)} \)  such that $\mathcal{A}^{(2)}\in \mathbb{C}^{S(s) \times T(t)}$ is an outer inverse of $\mathcal{A}$.
If $\mathcal{D}=\mathcal{A}+\mathcal{E}$, $\Vert \mathcal{E}*_t\mathcal{A}^{(2)}\Vert<1 $ and $\mathcal{E}=\mathcal{E}*_t\mathcal{A}^{(2)}*_s\mathcal{A}$ (or $\mathcal{E}=\mathcal{A} *_t\mathcal{A}^{(2)}*_s\mathcal{E}$), then
    $$\mathcal{D}^{(2)}=\mathcal{A}^{(2)}*_s\bigDrho     =\bigDdelta  *_t\mathcal{A}^{(2)};$$
    $$\mathcal{D}^{(2)} - \mathcal{A}^{(2)} =\mathcal{A}^{(2)}*_s (\bigDrho - \mathcal{I});$$
    $$\frac{\Vert\mathcal{D}^{(2)} - \mathcal{A}^{(2)}\Vert}{\Vert\mathcal{A}^{(2)}\Vert} \leq \frac{2 - \Vert\mathcal{E} *_t \mathcal{A}^{(2)}\Vert}{1 - \Vert\mathcal{E} *_t \mathcal{A}^{(2)} \Vert}.$$
\end{theorem}

\begin{proof} The assumption $\Vert \mathcal{E}*_t\mathcal{A}^{(2)}\Vert<1 $ implies the existence of $\bigDrho $ and $\bigDdelta  $.
Now, the identity $\mathcal{D}=(\mathcal{I}+\mathcal{E} *_t\mathcal{A}^{(2)})*_s\mathcal{A}$ gives
$$\mathcal{A}^{(2)}*_s\bigDrho*_s\mathcal{D}=
\mathcal{A}^{(2)}*_s\mathcal{A}.$$
Hence, we obtain
\begin{equation*}
\aligned
\mathcal{A}^{(2)}*_s\bigDrho  *_s \mathcal{D} *_t \mathcal{A}^{(2)}*_s\bigDrho  &=\mathcal{A}^{(2)}*_s\mathcal{A}*_t \mathcal{A}^{(2)}*_s\bigDrho \\
&=\mathcal{A}^{(2)}*_s\bigDrho,
\endaligned
\end{equation*}
which yield $\mathcal{D}^{(2)}=\mathcal{A}^{(2)}*_s\bigDrho$.

Since \begin{eqnarray*}
(\mathcal{I}+\mathcal{A}^{(2)} *_s\mathcal{E}) *_t \mathcal{A}^{(2)}
&=&\mathcal{A}^{(2)}+\mathcal{A}^{(2)}*_s \mathcal{E} *_t\mathcal{A}^{(2)}\\&=&\mathcal{A}^{(2)} *_s
(\mathcal{I}+\mathcal{E} *_t \mathcal{A}^{(2)}),
\end{eqnarray*}
we deduce that
$\mathcal{A}^{(2)}*_s\bigDrho      =\bigDdelta  *_t\mathcal{A}^{(2)}$.

Then, it follows
  $$ \mathcal{D}^{(2)} - \mathcal{A}^{(2)} =\mathcal{A}^{(2)}*_s (\bigDrho - \mathcal{I}).$$
From \eqref{0.4}, we bound the norm
$$\Vert \bigDrho \Vert \leq \frac{1}{1 - \Vert\mathcal{E} *_t \mathcal{A}^{(2)}\Vert}.$$
Thus
$$\Vert\mathcal{D}^{(2)}\Vert \leq \Vert\mathcal{A}^{(2)}\Vert \cdot \frac{1}{1 - \Vert\mathcal{E} *_t \mathcal{A}^{(2)}\Vert}.$$
Applying the triangle inequality
$$\Vert\mathcal{D}^{(2)} - \mathcal{A}^{(2)}\Vert \leq \Vert\mathcal{D}^{(2)}\Vert + \Vert\mathcal{A}^{(2)}\Vert \leq \Vert\mathcal{A}^{(2)}\Vert \left( \frac{1}{1 - \Vert\mathcal{E} *_t \mathcal{A}^{(2)}\Vert} + 1 \right)
$$
and simplifying the expression
$$\frac{\Vert\mathcal{D}^{(2)} - \mathcal{A}^{(2)}\Vert}{\Vert\mathcal{A}^{(2)}\Vert} \leq \frac{1}{1 - \Vert\mathcal{E} *_t \mathcal{A}^{(2)}\Vert} + 1 = \frac{2 - \Vert\mathcal{E} *_t \mathcal{A}^{(2)}\Vert}{1 - \Vert\mathcal{E} *_t \mathcal{A}^{(2)}\Vert}$$
we finalize the proof.
\end{proof}

Perturbation expressions for the $(\mathcal{B})$-inverse follow by the results of Theorem \ref{0.3B}.
We will use the notations
$\bigBrho=(\mathcal{I}+\mathcal{E} *_t\mathcal{A}_{\aunclfamily{\bf R}(\mathcal{B}),*}^{(2)})^{-1}$ and $\bigBdelta  =(\mathcal{I}+\mathcal{A}_{\aunclfamily{\bf R}(\mathcal{B}),*}^{(2)} *_s\mathcal{E})^{-1}$.

\begin{theorem}\label{0.4B}
    Let \( \mathcal{A} \in \mathbb{C}^{S(s) \times T(t)} \) and the $(\mathcal{B})$-inverse $\mathcal{A}_{\aunclfamily{\bf R}(\mathcal{B}),*}^{(2)}\in \mathbb{C}^{S(s) \times T(t)}$  exist.
If $\mathcal{D}=\mathcal{A}+\mathcal{E}$,  $\Vert \mathcal{E}*_t\mathcal{A}_{\aunclfamily{\bf R}(\mathcal{B}), *}^{(2)}\Vert<1 $ and $\mathcal{E}=\mathcal{E}*_t\mathcal{A}_{\aunclfamily{\bf R}(\mathcal{B}), *}^{(2)}*_s\mathcal{A}$,
then
$$\mathcal{D}_{\aunclfamily{\bf R}(\mathcal{B}), *}^{(2)}=\mathcal{A}_{\aunclfamily{\bf R}(\mathcal{B}), *}^{(2)}*_s \bigBrho    =\bigBdelta  *_t\mathcal{A}_{\aunclfamily{\bf R}(\mathcal{B}), *}^{(2)},$$
 %$$\mathcal{D}_{\aunclfamily{\bf R}(\mathcal{B}), *}^{(2)} - \mathcal{A}_{\aunclfamily{\bf R}(\mathcal{B}), *}^{(2)} =\mathcal{A}_{\aunclfamily{\bf R}(\mathcal{B}), *}^{(2)}*_s (\bigBrho - \mathcal{I} ).$$
 $$\frac{\Vert\mathcal{D}_{\aunclfamily{\bf R}(\mathcal{B}), *}^{(2)} - \mathcal{A}_{\aunclfamily{\bf R}(\mathcal{B}), *}^{(2)}\Vert}{\Vert\mathcal{A}_{\aunclfamily{\bf R}(\mathcal{B}), *}^{(2)}\Vert} \leq \frac{2 - \Vert\mathcal{E} *_t \mathcal{A}_{\aunclfamily{\bf R}(\mathcal{B}), *}^{(2)}\Vert}{1 - \Vert\mathcal{E} *_t \mathcal{A}_{\aunclfamily{\bf R}(\mathcal{B}), *}^{(2)} \Vert}.$$
\end{theorem}

\begin{proof} Based on Theorem \ref{0.3B},
$\mathcal{A}_{\aunclfamily{\bf R}(\mathcal{B}), *}^{(2)}*_s\bigBrho     =\bigBdelta   *_t\mathcal{A}_{\aunclfamily{\bf R}(\mathcal{B}), *}^{(2)}$ represents an outer inverse of $\mathcal{D}$.
Because $\aunclfamily{\bf R}(\mathcal{A}_{\aunclfamily{\bf R}(\mathcal{B}), *}^{(2)}*_s\bigBrho )=\aunclfamily{\bf R}(\mathcal{A}_{\aunclfamily{\bf R}(\mathcal{B}), *}^{(2)})=\aunclfamily{\bf R}(\mathcal{B})$, we conclude
$\mathcal{D}_{\aunclfamily{\bf R}(\mathcal{B}), *}^{(2)}=\mathcal{A}_{\aunclfamily{\bf R}(\mathcal{B}), *}^{(2)}*_s\bigBrho =\bigBdelta  *_t\mathcal{A}_{\aunclfamily{\bf R}(\mathcal{B}), *}^{(2)}$.
The remaining steps are simple and can be completed by applying Theorem \ref{0.3B}.
\end{proof}

Similarly, Theorem \ref{0.4C} gives corresponding results for outer inverses with specified null space.
We will use the notations
$\bigCrho =(\mathcal{I}+\mathcal{E} *_t\mathcal{A}_{*,\aunclfamily{\bf N}(\mathcal{C})}^{(2)})^{-1}$ and $\bigCdelta  =(\mathcal{I}+\mathcal{A}_{*,\aunclfamily{\bf N}(\mathcal{C})}^{(2)} *_s\mathcal{E})^{-1}$.

\begin{theorem}\label{0.4C}
    Let \( \mathcal{A} \in \mathbb{C}^{S(s) \times T(t)} \) and the $(\mathcal{C})$-inverse $\mathcal{A}_{*, \aunclfamily{\bf N}(\mathcal{C})}^{(2)}\in \mathbb{C}^{S(s) \times T(t)}$  exist. If $\mathcal{D}=\mathcal{A}+\mathcal{E}$, $\Vert \mathcal{E}*_t\mathcal{A}_{*, \aunclfamily{\bf N}(\mathcal{C})}^{(2)}\Vert<1 $ and $\mathcal{E}=\mathcal{A}*_t\mathcal{A}_{*, \aunclfamily{\bf N}(\mathcal{C})}^{(2)}*_s\mathcal{E}$, then
    $$\mathcal{D}_{*, \aunclfamily{\bf N}(\mathcal{C})}^{(2)}=\mathcal{A}_{*, \aunclfamily{\bf N}(\mathcal{C})}^{(2)}*_s\bigCrho =\bigCdelta  *_t\mathcal{A}_{*, \aunclfamily{\bf N}(\mathcal{C})}^{(2)},$$
    %$$\mathcal{D}_{*, \aunclfamily{\bf N}(\mathcal{C})}^{(2)} - \mathcal{A}_{*, \aunclfamily{\bf N}(\mathcal{C})}^{(2)} =\mathcal{A}_{*, \aunclfamily{\bf N}(\mathcal{C})}^{(2)}*_s (\bigCrho - \mathcal{I}).$$
    $$\frac{\Vert\mathcal{D}_{*, \aunclfamily{\bf N}(\mathcal{C})}^{(2)} - \mathcal{A}_{*, \aunclfamily{\bf N}(\mathcal{C})}^{(2)}\Vert}{\Vert\mathcal{A}_{*, \aunclfamily{\bf N}(\mathcal{C})}^{(2)}\Vert} \leq \frac{2 - \Vert\mathcal{E} *_t \mathcal{A}_{*, \aunclfamily{\bf N}(\mathcal{C})}^{(2)}\Vert}{1 - \Vert\mathcal{E} *_t \mathcal{A}_{*, \aunclfamily{\bf N}(\mathcal{C})}^{(2)} \Vert}.$$
\end{theorem}

The rest of the paper investigates perturbation of the $(\mathcal{B},\mathcal{C})$-inverse under appropriate restrictions.
The practical notations
$\bigBCrho =(\mathcal{I}+\mathcal{E} *_t\mathcal{A}_{\aunclfamily{\bf R}(\mathcal{B}),\aunclfamily{\bf N}(\mathcal{C})}^{(2)})^{-1}$ and $\bigBCdelta  =(\mathcal{I}+\mathcal{A}_{\aunclfamily{\bf R}(\mathcal{B}),\aunclfamily{\bf N}(\mathcal{C})}^{(2)}*_s\mathcal{E} )^{-1}$ will be used with the aim to improve presentation.

\begin{theorem}\label{0.3BC}
   If \( \mathcal{A} \in \mathbb{C}^{S(s) \times T(t)} \), the $(\mathcal{B},\mathcal{C})$-inverse $\mathcal{A}_{\aunclfamily{\bf R}(\mathcal{B}), \aunclfamily{\bf N}(\mathcal{C})}^{(2)}\in \mathbb{C}^{S(s) \times T(t)}$  exists, $\mathcal{D}=\mathcal{A}+\mathcal{E}$, $\Vert \mathcal{E}*_t\mathcal{A}_{\aunclfamily{\bf R}(\mathcal{B}), \aunclfamily{\bf N}(\mathcal{C})}^{(2)}\Vert<1 $ and $\mathcal{E}=\mathcal{E}*_t\mathcal{A}_{\aunclfamily{\bf R}(\mathcal{B}), \aunclfamily{\bf N}(\mathcal{C})}^{(2)}*_s\mathcal{A}=\mathcal{A}*_t\mathcal{A}_{\aunclfamily{\bf R}(\mathcal{B}),\aunclfamily{\bf N}(\mathcal{C})}^{(2)}*_s\mathcal{E}$, then
    $$\mathcal{D}_{\aunclfamily{\bf R}(\mathcal{B}), \aunclfamily{\bf N}(\mathcal{C})}^{(2)}=\mathcal{A}_{\aunclfamily{\bf R}(\mathcal{B}), \aunclfamily{\bf N}(\mathcal{C})}^{(2)}*_s\bigBCrho
    =\bigBCdelta  *_t\mathcal{A}_{\aunclfamily{\bf R}(\mathcal{B}), \aunclfamily{\bf N}(\mathcal{C})}^{(2)},$$
    %$$\mathcal{D}_{\aunclfamily{\bf R}(\mathcal{B}),\aunclfamily{\bf N}(\mathcal{C})}^{(2)} - \mathcal{A}_{\aunclfamily{\bf R}(\mathcal{B}),\aunclfamily{\bf N}(\mathcal{C})}^{(2)} =\mathcal{A}_{\aunclfamily{\bf R}(\mathcal{B}),\aunclfamily{\bf N}(\mathcal{C})}^{(2)}*_s (\bigBCrho - \mathcal{I}).$$
    $$\frac{\Vert\mathcal{D}_{\aunclfamily{\bf R}(\mathcal{B}),\aunclfamily{\bf N}(\mathcal{C})}^{(2)} - \mathcal{A}_{\aunclfamily{\bf R}(\mathcal{B}),\aunclfamily{\bf N}(\mathcal{C})}^{(2)}\Vert}{\Vert\mathcal{A}_{\aunclfamily{\bf R}(\mathcal{B}),\aunclfamily{\bf N}(\mathcal{C})}^{(2)}\Vert} \leq \frac{2 - \Vert\mathcal{E} *_t \mathcal{A}_{\aunclfamily{\bf R}(\mathcal{B}),\aunclfamily{\bf N}(\mathcal{C})}^{(2)}\Vert}{1 - \Vert\mathcal{E} *_t \mathcal{A}_{\aunclfamily{\bf R}(\mathcal{B}),\aunclfamily{\bf N}(\mathcal{C})}^{(2)} \Vert}.$$
\end{theorem}

Applying perturbation results for inner inverse given in Section 3, we will obtain perturbation formula for outer inverses.

\begin{theorem}\label{TheoremBInverse}
Let $\mathcal{A}, \mathcal{E}\in\mathbb{C}^{S(s)\times T(t)}$, $\mathcal{B}\in\mathbb{C}^{T(t)\times U(u)}$ and $\mathcal{A}^{(2)}_{\aunclfamily{\bf R}(\mathcal{B}),*}$ be the outer inverse of $\mathcal{A}$ with prescribed range $\mathcal{\aunclfamily{\bf R}(\mathcal{B})}$.
If the following conditions are met $\mathcal{D}=\mathcal{A}+\mathcal{E}$ with \\
1. $(\mathcal{D}^{(1)}*_s\mathcal{D}*_t\mathcal{B}*_u \mathcal{B}^{(1)})^2=\mathcal{D}^{(1)}*_s\mathcal{D}*_t\mathcal{B}*_u \mathcal{B}^{(1)}$;\\
2. $\rra{\mathcal{A}*_t \mathcal{B}}=\rra{\mathcal{D}*_t\mathcal{B}}=\rra{\mathcal{B}}$; \\
3. $\Vert \mathcal{A}^{(1)}\Vert \Vert \mathcal{E} \Vert <1 \ \  and \ \ \rra{\mathcal{A}}=\rra{\mathcal{D}}, $\\
then
\begin{equation}\label{The1.1}
\mathcal{D}^{(2)}_{\aunclfamily{\bf R}(\mathcal{B}),*}=\mathcal{B}*_u \mathcal{B}^{(1)}*_t\mathcal{A}^{(1)}*_s\bigrho
\end{equation}
%\begin{equation}\label{The1.2}
%\mathcal{D}^{(2)}_{\aunclfamily{\bf R}(\mathcal{B}),*}-\mathcal{A}^{(2)}_{\aunclfamily{\bf R}(\mathcal{B}),*}=\mathcal{B}*_u \mathcal{B}^{(1)}*_t\mathcal{A}^{(1)}*_t\bigrho -\mathcal{B}*_u (\mathcal{A}*_t\mathcal{B})^{(1)},
%\end{equation}
\begin{equation}\label{The1.3}
\frac{\Vert \mathcal{D}^{(2)}_{\aunclfamily{\bf R}(\mathcal{B}),*}-\mathcal{A}^{(2)}_{\aunclfamily{\bf R}(\mathcal{B}),*}\Vert}{\Vert \mathcal{A}^{(2)}_{\aunclfamily{\bf R}(\mathcal{B}),*}\Vert} \leq
\frac{\Vert \mathcal{B}*_u \mathcal{B}^{(1)} \Vert \Vert \mathcal{A}^{(1)} \Vert}{(1-\Vert \mathcal{E}*_t\mathcal{A}^{(1)} \Vert)\Vert \mathcal{B}*_u (\mathcal{A}*_t\mathcal{B})^{(1)} \Vert}+1.
\end{equation}
\end{theorem}
\begin{proof}
By the condition 2 and \eqref{0.31}, $\mathcal{D}^{(2)}_{\aunclfamily{\bf R}(\mathcal{B}),*}$ can be written as:
\begin{equation}\label{Proof1.1}
\mathcal{D}^{(2)}_{\aunclfamily{\bf R}(\mathcal{B}),*}=\mathcal{B}*_u (\mathcal{D}*_t\mathcal{B})^{(1)}.
\end{equation}
From $\mathcal{D}=\mathcal{A}+\mathcal{E}$, we get:
\begin{equation}\label{Proof1.2}
\mathcal{D}*_t\mathcal{B}=(\mathcal{A}+\mathcal{E})*_t\mathcal{B}=\mathcal{A}*_t\mathcal{B}+\mathcal{E}*_t\mathcal{B}.
\end{equation}
Lemma \ref{le-inv} guarantees that $\mathcal{I}+\mathcal{A}^{(1)}*_s\mathcal{E}$ is invertible.
Since $\rra{\mathcal{A}}=\rra{\mathcal{D}}$, by to Theorem \ref{Chai1Inverse1Trank}, an inner inverse of $\mathcal{D}=\mathcal{A}+\mathcal{E}$ is equal to
\begin{equation}\label{Proof1.3}
\mathcal{D}^{(1)}=\mathcal{A}^{(1)}*_s \bigrho.
\end{equation}
By considering equations 1 and 2, along with \eqref{0.1}, and substituting \(\mathcal{D}^{(1)}\) from \label{Proof1.3} into \((\mathcal{D} * _n \mathcal{B})^{(1)}\) gives us the following result:
\begin{equation}\label{Proof1.4}
(\mathcal{D}*_t\mathcal{B})^{(1)}=\mathcal{B}^{(1)}*_t\mathcal{D}^{(1)}=\mathcal{B}^{(1)}*_t\mathcal{A}^{(1)}*_s\bigrho.
\end{equation}
Thus,
\begin{equation}\label{Proof1.5}
\mathcal{D}^{(2)}_{\aunclfamily{\bf R}(\mathcal{B}),*}=\mathcal{B}*_u (\mathcal{D}*_t\mathcal{B})^{(1)}=\mathcal{B}*_u \mathcal{B}^{(1)}*_t\mathcal{D}^{(1)}=\mathcal{B}*_u \mathcal{B}^{(1)}*_t\mathcal{A}^{(1)}*_s\bigrho.
\end{equation}
By \eqref{0.31} and \eqref{Proof1.5}, it follows $\mathcal{A}^{(2)}_{\aunclfamily{\bf R}(\mathcal{B}),*}=\mathcal{B}*_u (\mathcal{A}*_t\mathcal{B})^{(1)}$, and further
\begin{equation}\label{Proof1.6}
\mathcal{D}^{(2)}_{\aunclfamily{\bf R}(\mathcal{B}),*}-\mathcal{A}^{(2)}_{\aunclfamily{\bf R}(\mathcal{B}),*}=\mathcal{B}*_u \mathcal{B}^{(1)}*_t\mathcal{A}^{(1)}*_s\bigrho-\mathcal{B}*_u (\mathcal{A}*_t\mathcal{B})^{(1)}.
\end{equation}
The application of the norm to equation \eqref{Proof1.6} results in the following outcome:
\begin{equation}\label{Proof1.7}
\Vert \mathcal{D}^{(2)}_{\aunclfamily{\bf R}(\mathcal{B}),*}-\mathcal{A}^{(2)}_{\aunclfamily{\bf R}(\mathcal{B}),*} \Vert \leq \Vert \mathcal{B}*_u \mathcal{B}^{(1)}\Vert\Vert \mathcal{A}^{(1)}\Vert\Vert \bigrho\Vert +\Vert \mathcal{B}*_u (\mathcal{A}*_t\mathcal{B})^{(1)}\Vert.
\end{equation}
From \eqref{0.4} it can be obtained
\begin{equation}\label{Proof1.8}
\Vert \bigrho\Vert \leq \frac{1}{1-\Vert \mathcal{E}*_t\mathcal{A}^{(1)}\Vert}.
\end{equation}
Substituting the expression from \eqref{Proof1.8} into \eqref{Proof1.7} yields
\begin{equation}\label{Proof1.9}
\Vert \mathcal{D}^{(2)}_{\aunclfamily{\bf R}(\mathcal{B}),*}-\mathcal{A}^{(2)}_{\aunclfamily{\bf R}(\mathcal{B}),*} \Vert \leq \frac{\Vert \mathcal{B}*_u \mathcal{B}^{(1)}\Vert\Vert \mathcal{A}^{(1)}\Vert}{1-\Vert\mathcal{E}*_t\mathcal{A}^{(1)}\Vert}+\Vert \mathcal{B}*_u (\mathcal{A}*_t\mathcal{B})^{(1)}\Vert.
\end{equation}
Normalization by $\Vert \mathcal{A}^{(2)}_{\aunclfamily{\bf R}(\mathcal{B}),*}\Vert$ results in
\begin{equation}\label{Proof1.10}
\frac{\Vert \mathcal{D}^{(2)}_{\aunclfamily{\bf R}(\mathcal{B}),*}-\mathcal{A}^{(2)}_{\aunclfamily{\bf R}(\mathcal{B}),*}\Vert}{\Vert \mathcal{A}^{(2)}_{\aunclfamily{\bf R}(\mathcal{B}),*}\Vert}
\leq \frac{\Vert \mathcal{B}*_u \mathcal{B}^{(1)}\Vert\Vert \mathcal{A}^{(1)}\Vert}{(1-\Vert \mathcal{E}*_t\mathcal{A}^{(1)}\Vert)\Vert \mathcal{B}*_u (\mathcal{A}*_t\mathcal{B})^{(1)}\Vert}+1,
\end{equation}
which ends the proof.
\end{proof}

\begin{example}
Let $\mathcal{A}\in\mathbb{R}^{(3\times 2\times 1)\times (2\times 3)}$ be defined by
\begin{equation*}
\aligned
\mathcal{A}(:,:,1,1,1)&=\begin{bmatrix} \ 1 & 0  \ \\ \ 1 & 0 \ \\ \ 0 & 0 \  \end{bmatrix};\
\mathcal{A}(:,:,1,2,1)=\begin{bmatrix} \ 1 & 0  \ \\ \ 1 & 0 \ \\ \ 0 & 0 \  \end{bmatrix};\
\mathcal{A}(:,:,1,1,2)=\begin{bmatrix} \ 0 & 1  \ \\ \ 0 & 0 \ \\ \ 1 & 0 \  \end{bmatrix};\\
\mathcal{A}(:,:,1,2,2)&=\begin{bmatrix} \ 0 & 1  \ \\ \ 0 & 0 \ \\ \ 1 & 0 \  \end{bmatrix}; \
\mathcal{A}(:,:,1,1,3)=\begin{bmatrix} \ 0 & 0  \ \\ \ 0 & 0 \ \\ \ 0 & 0 \  \end{bmatrix};\
\mathcal{A}(:,:,1,2,3)=\begin{bmatrix} \ 0 & 0  \ \\ \ 0 & 0 \ \\ \ 0 & 0 \  \end{bmatrix},
\endaligned
\end{equation*}
and $\mathcal{B}\in\mathbb{R}^{(2\times 3)\times (2\times 2)}$ by
 \begin{center}
$\mathcal{B}(:,:,1,1)=\begin{bmatrix} \ 5 & 0 & 0 \ \\ \ 5 & 0 & 0 \ \end{bmatrix};$
$\mathcal{B}(:,:,1,2)=\begin{bmatrix} \ 3 & 0 & 0 \ \\ \ 3 & 0 & 0 \  \end{bmatrix};$\\
$\mathcal{B}(:,:,2,1)=\begin{bmatrix} \ 0 & 2 & 0 \ \\ \ 0 & 2 & 0 \  \end{bmatrix};$
$\mathcal{B}(:,:,2,2)=\begin{bmatrix} \ 0 & 1 & 0 \ \\ \ 0 & 1 & 0 \  \end{bmatrix}$.
\end{center}
Consider perturbation tensor $\mathcal{E}\in\mathbb{R}^{(3\times 2\times 1)\times (2\times 3)}$ with entries
\begin{equation*}
\aligned
\mathcal{E}(:,:,1,1,1)&=\begin{bmatrix} \ 0.1 & 0  \ \\ \ 0.1 & 0 \ \\ \ 0 & 0 \ \end{bmatrix};\
\mathcal{E}(:,:,1,2,1)=\begin{bmatrix} \ 0.1 & 0  \ \\ \ 0.1 & 0 \ \\ \ 0 & 0 \ \end{bmatrix};\
\mathcal{E}(:,:,1,1,2)=\begin{bmatrix} \ 0 & 0.1  \ \\ \ 0 & 0 \ \\ \ 0 & 0 \ \end{bmatrix};\\
\mathcal{E}(:,:,1,2,2)&=\begin{bmatrix} \ 0 & 0.1  \ \\ \ 0 & 0 \ \\ \ 0 & 0 \ \end{bmatrix};\
\mathcal{E}(:,:,1,1,3)=\begin{bmatrix} \ 0 & 0  \ \\ \ 0 & 0 \ \\ \ 0 & 0 \  \end{bmatrix};\
\mathcal{E}(:,:,1,2,3)=\begin{bmatrix} \ 0 & 0  \ \\ \ 0 & 0 \ \\ \ 0 & 0 \  \end{bmatrix}.
\endaligned
\end{equation*}
Next, Algorithm 4 from \cite{LMA685} is used to compute $\mathcal{A}^{(1)}$.
The F-norms of $\mathcal{A}^{(1)}$ and $\mathcal{E}$ are equal to
\begin{center}
$\Vert \mathcal{A}^{(1)} \Vert_F=0.7071,\ \Vert \mathcal{E} \Vert_F=0.2449,\ \Vert \mathcal{A}^{(1)}\Vert_F\Vert \mathcal{E} \Vert_F=0.1732<1$.
\end{center}
Then, Algorithm 3 form \cite{LMA685} is used to verify
$\rra{\mathcal{A}*_2 \mathcal{B}}=\rra{\mathcal{D}*_2\mathcal{B}}=\rra{\mathcal{B}}=2$ and $\rra{\mathcal{A}}=\rra{\mathcal{D}}=2$.
Indeed, after performing the necessary calculations, it can be verified that $ \mathcal{D}^{(2)}_{\aunclfamily{\bf R}(\mathcal{B}),*} \in \mathbb{R}^{(2\times 3)\times(3\times2\times 1)} $ is expressed as
 \begin{center}
$\mathcal{D}^{(2)}_{\aunclfamily{\bf R}(\mathcal{B}),*}(:,:,1,1,1)=\begin{bmatrix} \ 0.2273 & 0 & 0 \ \\ \ 0.2273 & 0 & 0 \ \end{bmatrix};$
$\mathcal{D}^{(2)}_{\aunclfamily{\bf R}(\mathcal{B}),*}(:,:,2,1,1)=\begin{bmatrix}\  0.2273 & 0 & 0 \ \\ \ 0.2273 & 0 & 0 \ \end{bmatrix};$
$\mathcal{D}^{(2)}_{\aunclfamily{\bf R}(\mathcal{B}),*}(:,:,3,1,1)=\begin{bmatrix} \ 0 & 0.2381 & 0 \ \\ \ 0 & 0.2381 & 0 \ \end{bmatrix};$
$\mathcal{D}^{(2)}_{\aunclfamily{\bf R}(\mathcal{B}),*}(:,:,1,2,1)=\begin{bmatrix}\ 0 & 0.2381 & 0 \ \\ \ 0 & 0.2381 & 0 \ \end{bmatrix};$
$\mathcal{D}^{(2)}_{\aunclfamily{\bf R}(\mathcal{B}),*}(:,:,2,2,1)=\begin{bmatrix} \ 0 & 0 & 0 \ \\ \ 0 & 0 & 0 \ \end{bmatrix};$
$\mathcal{D}^{(2)}_{\aunclfamily{\bf R}(\mathcal{B}),*}(:,:,3,2,1)=\begin{bmatrix} \ 0 & 0 & 0 \ \\ \ 0 & 0 & 0 \ \end{bmatrix}$,
 \end{center}
 and  $ \mathcal{A}^{(2)}_{\aunclfamily{\bf R}(\mathcal{B}),*}\in \mathbb{R}^{ (2\times 3)\times(3\times2\times 1)} $ is given by
\begin{equation*}
\aligned
\mathcal{A}^{(2)}_{\aunclfamily{\bf R}(\mathcal{B}),*}(:,:,1,1,1)&=\begin{bmatrix} \ 0.25 & 0 & 0 \ \\ \ 0.25 & 0 & 0 \ \end{bmatrix};\
\mathcal{A}^{(2)}_{\aunclfamily{\bf R}(\mathcal{B}),*}(:,:,2,1,1)=\begin{bmatrix}\ 0.25 & 0 & 0 \ \\ \ 0.25 & 0 & 0 \ \end{bmatrix};\\
\mathcal{A}^{(2)}_{\aunclfamily{\bf R}(\mathcal{B}),*}(:,:,3,1,1)&=\begin{bmatrix} \ 0 & 0.25 & 0 \ \\ \ 0 & 0.25 & 0 \ \end{bmatrix};\
\mathcal{A}^{(2)}_{\aunclfamily{\bf R}(\mathcal{B}),*}(:,:,1,2,1)=\begin{bmatrix}\ 0 & 0.25 & 0 \ \\ \ 0 & 0.25 & 0 \ \end{bmatrix};\\
\mathcal{A}^{(2)}_{\aunclfamily{\bf R}(\mathcal{B}),*}(:,:,2,2,1)&=\begin{bmatrix} \ 0 & 0 & 0 \ \\ \ 0 & 0 & 0 \ \end{bmatrix};\
\mathcal{A}^{(2)}_{\aunclfamily{\bf R}(\mathcal{B}),*}(:,:,3,2,1)=\begin{bmatrix} \ 0 & 0 & 0 \ \\ \ 0 & 0 & 0 \ \end{bmatrix}.
\endaligned
\end{equation*}
By calculation, $(\mathcal{D}^{(1)}*_3\mathcal{D}*_2\mathcal{B}*_2\mathcal{B}^{(1)})$  is idempotent.
Conditions are met, we proceed with the computation
\begin{center}
 $\Vert \mathcal{D}^{(2)}_{\aunclfamily{\bf R}(\mathcal{B}),*}-\mathcal{A}^{(2)}_{\aunclfamily{\bf R}(\mathcal{B}),*}\Vert_F=0.0513$, \ \ \ \
$\Vert \mathcal{E}*_2\mathcal{A}^{(1)} \Vert_F=0.1225$, \ \ \ \
$\Vert \mathcal{B}*_2\mathcal{B}^{(1)} \Vert_F=1.4142$, \ \\
 $\Vert \mathcal{A}^{(2)}_{\aunclfamily{\bf R}(\mathcal{B}),*}\Vert_F = \Vert \mathcal{B}*_2(\mathcal{A}*_2\mathcal{B})^{(1)} \Vert_F=0.7071$.
 \end{center}
Further calculation yields
$$\frac{\Vert \mathcal{D}^{(2)}_{\aunclfamily{\bf R}(\mathcal{B}),*}-\mathcal{A}^{(2)}_{\aunclfamily{\bf R}(\mathcal{B}),*}\Vert_F}{\Vert \mathcal{A}^{(2)}_{\aunclfamily{\bf R}(\mathcal{B}),*}\Vert_F}=0.0726,$$
$$\frac{\Vert \mathcal{B}*_2\mathcal{B}^{(1)} \Vert_F \Vert \mathcal{A}^{(1)} \Vert_F}{(1-\Vert \mathcal{E}*_2\mathcal{A}^{(1)} \Vert_F)\Vert \mathcal{B}*_2(\mathcal{A}*_2\mathcal{B})^{(1)} \Vert_F}+1=2.6116.$$
Therefore, the inequality of Theorem \ref{TheoremBInverse} holds.
\end{example}

\begin{corollary}
Based on the condition of Theorem \ref{TheoremBInverse}, if $(\mathcal{A}*_t\mathcal{B})^{(1)}=\mathcal{B}^{(1)}*_t\mathcal{A}^{(1)}$, the following inequality holds:
$$\frac{\Vert \mathcal{D}^{(2)}_{\aunclfamily{\bf R}(\mathcal{B}),*}-\mathcal{A}^{(2)}_{\aunclfamily{\bf R}(\mathcal{B}),*}\Vert}{\Vert \mathcal{A}^{(2)}_{\aunclfamily{\bf R}(\mathcal{B}),*}\Vert} \leq \frac {1}{1-\Vert \mathcal{E}*_t\mathcal{A}^{(1)} \Vert}+1.$$
\end{corollary}

\begin{theorem}\label{TheoremCInverse}
Let $\mathcal{A}, \mathcal{E}\in\mathbb{C}^{S(s)\times T(t)}$, $\mathcal{C}\in\mathbb{C}^{V(v)\times S(s)}$ and $\mathcal{A}^{(2)}_{*,\aunclfamily{\bf N}(\mathcal{C})}$ be an outer inverse of $\mathcal{A}$ possessing prescribed null space $\aunclfamily{\bf N}(\mathcal{C})$.
If the following conditions hold for $\mathcal{D}=\mathcal{A}+\mathcal{E}$, it can be concluded \\
1. $(\mathcal{C}^{(1)}*_v\mathcal{C}*_s\mathcal{D}*_t\mathcal{D}^{(1)})^2=\mathcal{C}^{(1)}*_v\mathcal{C}*_s\mathcal{D}*_t\mathcal{D}^{(1)}$;\\
2. $\rra{\mathcal{C}*_s\mathcal{D}}=\rra{\mathcal{C}*_s\mathcal{A}}=\rra{\mathcal{C}}$; \\
3. $\Vert \mathcal{A}^{(1)}\Vert \Vert \mathcal{E} \Vert <1 \ \  and \ \ \rra{\mathcal{A}}=\rra{\mathcal{D}},$\\
then\\
\begin{equation}\label{The2.1}
\mathcal{D}^{(2)}_{*,\aunclfamily{\bf N}(\mathcal{C})}=\mathcal{A}^{(1)}*_s \bigrho*_s\mathcal{C}^{(1)}*_v\mathcal{C},
\end{equation}
%\begin{equation}\label{The2.2}
%\mathcal{D}^{(2)}_{*,N(\mathcal{C})}-\mathcal{A}^{(2)}_{*,N(\mathcal{C} )}=\mathcal{A}^{(1)}*_s \bigrho *_s \mathcal{C}^{(1)}*_v \mathcal{C} -  (\mathcal{C}*_s \mathcal{A} )^{(1)} *_v \mathcal{C},
%\end{equation}
\begin{equation}\label{The3.3}
\frac{\Vert \mathcal{D}^{(2)}_{*,\aunclfamily{\bf N}(\mathcal{C})}-\mathcal{A}^{(2)}_{*,\aunclfamily{\bf N}(\mathcal{C})}\Vert}{\Vert \mathcal{A}^{(2)}_{*,\aunclfamily{\bf N}(\mathcal{C})}\Vert} \leq
\frac{\Vert \mathcal{A}^{(1)} \Vert \Vert \mathcal{C}^{(1)}*_v \mathcal{C} \Vert}{(1-\Vert \mathcal{E}*_t\mathcal{A}^{(1)} \Vert) \Vert  (\mathcal{C}*_s \mathcal{A})^{(1)} *_v \mathcal{C} \Vert}+1.
\end{equation}
\end{theorem}

\begin{proof}
According to condition 2 and \eqref{0.32}, $\mathcal{D}^{(2)}_{*,\aunclfamily{\bf N}(\mathcal{C})}$ can be  expressed as follows:
\begin{equation}\label{Proof2.1}
\mathcal{D}^{(2)}_{*,\aunclfamily{\bf N}(\mathcal{C})}= (\mathcal{C}*_s \mathcal{D} )^{(1)} *_v \mathcal{C}.
\end{equation}
From $\mathcal{D}=\mathcal{A}+\mathcal{E}$, we can derive valuable insight:
\begin{equation}\label{Proof2.2}
\mathcal{C}*_s \mathcal{D} =\mathcal{C}*_s \mathcal{A} + \mathcal{C}*_s \mathcal{E}.
\end{equation}
Lemma \ref{le-inv} yields invertibility of $\mathcal{I}+\mathcal{E}*_t\mathcal{A}^{(1)} $.
Next $\rra{\mathcal{A}}=\rra{\mathcal{D}}$, and calculation based on Theorem \ref{Chai1Inverse1Trank} yields the inner inverse of $\mathcal{D}=\mathcal{A}+\mathcal{E}$ as follows:
\begin{equation}\label{Proof2.3}
\mathcal{D}^{(1)}=\mathcal{A}^{(1)}*_s \bigrho.
\end{equation}
By conditions $1$,$2$ and \eqref{0.1}, substitution of $\mathcal{D}^{(1)}$ from \eqref{Proof2.3} into $(\mathcal{C}*_s \mathcal{D})^{(1)}$ yields
\begin{equation}\label{Proof2.4}
(\mathcal{C}*_s \mathcal{D} )^{(1)}=  \mathcal{D}^{(1)}*_s \mathcal{C}^{(1)}.
\end{equation}
Thus,
\begin{equation}\label{Proof2.5}
\mathcal{D}^{(2)}_{*,\aunclfamily{\bf N}(\mathcal{C})}= ( \mathcal{D}^{(1)}*_s \mathcal{C}^{(1)} )*_u  \mathcal{C} = \mathcal{A}^{(1)}*_s \bigrho *_s \mathcal{C}^{(1)}*_v \mathcal{C}.
\end{equation}
From \eqref{0.32} and \eqref{Proof2.5}, $\mathcal{A}^{(2)}_{*,\aunclfamily{\bf N}(\mathcal{C})}= (\mathcal{C}*_s \mathcal{A})^{(1)} *_v \mathcal{C}$, then:

\begin{equation}\label{Proof2.6}
\mathcal{D}^{(2)}_{*,\aunclfamily{\bf N}(\mathcal{C})}-\mathcal{A}^{(2)}_{*,\aunclfamily{\bf N}(\mathcal{C})}=
\mathcal{A}^{(1)}*_s \bigrho *_s \mathcal{C}^{(1)}*_v \mathcal{C} -  (\mathcal{C}*_s \mathcal{A} )^{(1)} *_v \mathcal{C}.
\end{equation}
Applying the norm to both sides of equation \eqref{Proof2.6} yields the following results:
\begin{equation}\label{Proof2.7}
\Vert \mathcal{D}^{(2)}_{*,\aunclfamily{\bf N}(\mathcal{C})}-\mathcal{A}^{(2)}_{*,\aunclfamily{\bf N}(\mathcal{C})} \Vert \leq \Vert \mathcal{A}^{(1)}\Vert \Vert \bigrho \Vert \Vert \mathcal{C}^{(1)}*_v \mathcal{C}\Vert + \Vert(\mathcal{C}*_s \mathcal{A})^{(1)} *_v \mathcal{C} \Vert.
\end{equation}
An application of \eqref{0.4} leads to
\begin{equation}\label{Proof2.8}
\Vert \bigrho\Vert \leq \frac{1}{1-\Vert \mathcal{E}*_t\mathcal{A}^{(1)}\Vert}.
\end{equation}
The substitution of \eqref{Proof2.8} in \eqref{Proof2.7} yields
\begin{equation}\label{Proof2.9}
\Vert \mathcal{D}^{(2)}_{*,\aunclfamily{\bf N}(\mathcal{C})}-\mathcal{A}^{(2)}_{*,\aunclfamily{\bf N}(\mathcal{C})} \Vert \leq \frac{ \Vert \mathcal{A}^{(1)} \Vert \Vert \mathcal{C}^{(1)}*_v \mathcal{C} \Vert}{1-\Vert \mathcal{E}*_t\mathcal{A}^{(1)}  \Vert}+\Vert  (\mathcal{C}*_s \mathcal{A} )^{(1)} *_v \mathcal{C} \Vert.
\end{equation}
A normalization by $\Vert \mathcal{A}^{(2)}_{*,\aunclfamily{\bf N}(\mathcal{C})}\Vert $ initiates
\begin{equation*}\label{Proof2.10}
\frac{\Vert \mathcal{D}^{(2)}_{*,\aunclfamily{\bf N}(\mathcal{C})}-\mathcal{A}^{(2)}_{*,\aunclfamily{\bf N}(\mathcal{C})}\Vert}{\Vert \mathcal{A}^{(2)}_{*,\aunclfamily{\bf N}(\mathcal{C})}\Vert}
\leq \frac{\Vert \mathcal{A}^{(1)} \Vert \Vert \mathcal{C}^{(1)}*_v \mathcal{C} \Vert}{(1-\Vert \mathcal{E}*_t\mathcal{A}^{(1)} \Vert) \Vert  (\mathcal{C}*_s \mathcal{A})^{(1)} *_v \mathcal{C} \Vert}+1,
\end{equation*}
which confirms our original intention.
\end{proof}

\begin{example}
Suppose that  $\mathcal{A}\in\mathbb{R}^{(2\times 3)\times (3\times2\times 1)}$ is given with the entries
\begin{center}
$\mathcal{A}(:,:,1,1,1)=\begin{bmatrix} \ 1 & 0 & 0 \ \\ \ 0 & 0 & 0 \  \end{bmatrix}$;
$\mathcal{A}(:,:,2,1,1)=\begin{bmatrix} \ 0 & 0 & 0 \ \\ \ 1 & 0 & 0 \  \end{bmatrix}$;
$\mathcal{A}(:,:,3,1,1)=\begin{bmatrix} \ 1 & 0 & 0 \ \\ \ 0 & 0 & 0 \  \end{bmatrix}$;
$\mathcal{A}(:,:,1,2,1)=\begin{bmatrix} \ 0 & 0 & 0 \ \\ \ 1 & 0 & 0 \  \end{bmatrix}$;
$\mathcal{A}(:,:,2,2,1)=\begin{bmatrix} \ 1 & 0 & 0 \ \\ \ 0 & 0 & 0 \  \end{bmatrix}$;
$\mathcal{A}(:,:,3,2,1)=\begin{bmatrix} \ 0 & 0 & 0 \ \\ \ 1 & 0 & 0 \  \end{bmatrix}$.
\end{center}
and consider $\mathcal{C}\in\mathbb{R}^{(2\times 2)\times (2\times 3)}$ with components
 \begin{center}
$\mathcal{C}(:,:,1,1)=\begin{bmatrix} \ 3 & 3  \ \\ \ 0 & 0  \  \end{bmatrix}$;
$\mathcal{C}(:,:,1,2)=\begin{bmatrix} \ 0 & 0  \ \\ \ 0 & 0  \  \end{bmatrix}$;
$\mathcal{C}(:,:,1,3)=\begin{bmatrix} \ 0 & 0  \ \\ \ 0 & 0  \  \end{bmatrix}$;\\
$\mathcal{C}(:,:,2,1)=\begin{bmatrix} \ 0 & 2  \ \\ \ 2 & 0  \  \end{bmatrix}$;
$\mathcal{C}(:,:,2,2)=\begin{bmatrix} \ 0 & 0  \ \\ \ 0 & 0  \  \end{bmatrix}$;
$\mathcal{C}(:,:,2,3)=\begin{bmatrix} \ 0 & 0  \ \\ \ 0 & 0  \  \end{bmatrix}$.
\end{center}
Let the perturbation $\mathcal{E}\in\mathbb{R}^{(2\times 3)\times (3\times2\times1)}$ be with
\begin{center}
$\mathcal{E}(:,:,1,1,1)=\begin{bmatrix} \ 0 & 0 & 0 \ \\ \ 0 & 0 & 0 \ \end{bmatrix}$;
$\mathcal{E}(:,:,2,1,1)=\begin{bmatrix} \ 0 & 0 & 0 \ \\ \ 0 & 0 & 0 \ \end{bmatrix}$;
$\mathcal{E}(:,:,3,1,1)=\begin{bmatrix} \ 0.5 & 0 & 0 \ \\ \ 0.5 & 0 & 0 \ \end{bmatrix}$;
$\mathcal{E}(:,:,1,2,1)=\begin{bmatrix} \ 0 & 0 & 0 \ \\ \ 0 & 0 & 0 \ \end{bmatrix}$;
$\mathcal{E}(:,:,2,2,1)=\begin{bmatrix} \ 0 & 0 & 0 \ \\ \ 0 & 0 & 0 \ \end{bmatrix}$;
$\mathcal{E}(:,:,3,2,1)=\begin{bmatrix} \ 0 & 0 & 0 \ \\ \ 0 & 0 & 0 \ \end{bmatrix}$.
\end{center}
Next, we will apply Algorithm 4 from \cite{LMA685} to compute $\mathcal{A}^{(1)}$.
The F-norms of $\mathcal{A}^{(1)}$ and $\mathcal{E}$ are equal to
\begin{center}
$\Vert \mathcal{A}^{(1)} \Vert_F=0.8165$,\ \ \ $\Vert \mathcal{E} \Vert_F=0.7071$,\ \ \ $\Vert \mathcal{A}^{(1)}\Vert_F\Vert \mathcal{E} \Vert_F=0.5774<1$.
\end{center}
Then, Algorithm 3 form \cite{LMA685} can be used to verify $\rra{\mathcal{C}*_2\mathcal{D}}=\rra{\mathcal{C}*_2\mathcal{A}}=\rra{\mathcal{C}}=2$ and $\rra{\mathcal{A}}=\rra{\mathcal{D}}=2$.
Indeed, by computations, it follows
$\mathcal{D}^{(2)}_{*,\aunclfamily{\bf N}(\mathcal{C})}\in \mathbb{R}^{ (3\times2\times 1)\times(2\times 3)} $is given by
 \begin{center}
$\mathcal D^{(2)}_{*,\mathcal N(\mathcal C)}(:,:,1,1,1) = \begin{bmatrix} \ 0.2857 & -0.0476 \\ -0.0476 & 0.2857 \\ 0.2857 & -0.0476\ \end{bmatrix}; \quad \mathcal D^{(2)}_{*,\mathcal N(\mathcal C)}(:,:,1,2,1) = \begin{bmatrix} \ 0 & 0.3333 \\ 0.3333 & 0 \\ 0 & 0.3333\ \end{bmatrix};$
     $\mathcal D^{(2)}_{*,\mathcal N(\mathcal C)}(:,:,1,1,2) = \begin{bmatrix} \ 0 & 0 \\ 0 & 0 \\ 0 & 0\ \end{bmatrix}; \quad \mathcal D^{(2)}_{*,\mathcal N(\mathcal C)}(:,:,1,2,2) = \begin{bmatrix} \ 0 & 0 \\ 0 & 0 \\ 0 & 0\ \end{bmatrix};$
     $\mathcal D^{(2)}_{*,\mathcal N(\mathcal C)}(:,:,1,1,3) = \begin{bmatrix} \ 0 & 0 \\ 0 & 0 \\ 0 & 0\ \end{bmatrix}; \quad \mathcal D^{(2)}_{*,\mathcal N(\mathcal C)}(:,:,1,2,3) = \begin{bmatrix} \ 0 & 0 \\ 0 & 0 \\ 0 & 0\ \end{bmatrix}$.

 \end{center}
 and  $\mathcal{A}^{(2)}_{*,\aunclfamily{\bf N}(\mathcal{C})}\in \mathbb{R}^{ (3\times2\times 1)\times(2\times 3)} $is given by
  \begin{center}
$\mathcal A^{(2)}_{*,\mathcal N(\mathcal C)}(:,:,1,1,1) = \begin{bmatrix} \ 0.3333 & 0 \\ 0 & 0.3333 \\ 0.3333 & 0\ \end{bmatrix}; \quad \mathcal A^{(2)}_{*,\mathcal N(\mathcal C)}(:,:,1,2,1) = \begin{bmatrix} \ 0 & 0.3333 \\ 0.3333 & 0 \\ 0 & 0.3333\ \end{bmatrix};$
     $\mathcal A^{(2)}_{*,\mathcal N(\mathcal C)}(:,:,1,1,2) = \begin{bmatrix} \ 0 & 0 \\ 0 & 0 \\ 0 & 0\ \end{bmatrix}; \quad \mathcal A^{(2)}_{*,\mathcal N(\mathcal C)}(:,:,1,2,2) = \begin{bmatrix} \ 0 & 0 \\ 0 & 0 \\ 0 & 0\ \end{bmatrix};$
     $\mathcal A^{(2)}_{*,\mathcal N(\mathcal C)}(:,:,1,1,3) = \begin{bmatrix} \ 0 & 0 \\ 0 & 0 \\ 0 & 0\ \end{bmatrix}; \quad \mathcal A^{(2)}_{*,\mathcal N(\mathcal C)}(:,:,1,2,3) = \begin{bmatrix} \ 0 & 0 \\ 0 & 0 \\ 0 & 0\ \end{bmatrix}$.
 \end{center}
 By calculation, $(\mathcal{C}^{(1)}*_2\mathcal{C}*_2\mathcal{D}*_3\mathcal{D}^{(1)})$  is idempotent.
The necessary conditions have been met, allowing for the calculations to proceed as follows:
 \begin{center}
 $\Vert \mathcal{D}^{(2)}_{*,\aunclfamily{\bf N}(\mathcal{C})}-\mathcal{A}^{(2)}_{*,\aunclfamily{\bf N}(\mathcal{C})}\Vert_F=0.1116$, \ \ \
$\Vert \mathcal{E}*_3\mathcal{A}^{(1)} \Vert_F=0.2357$, \ \ \
 $\Vert \mathcal{C}^{(1)}*_2 \mathcal{C} \Vert_F=1.4142$, \\
$\Vert \mathcal{A}^{(2)}_{*,\aunclfamily{\bf N}(\mathcal{C})}\Vert_F=\Vert  (\mathcal{C}*_2 \mathcal{A})^{(1)} *_2 \mathcal{C} \Vert_F=0.8165$.
 \end{center}
In this situation, normalization results in
$$\frac{\Vert \mathcal{D}^{(2)}_{*,\aunclfamily{\bf N}(\mathcal{C})}-\mathcal{A}^{(2)}_{*,\aunclfamily{\bf N}(\mathcal{C})}\Vert_F}{\Vert \mathcal{A}^{(2)}_{*,\aunclfamily{\bf N}(\mathcal{C})}\Vert_F}=0.1429$$
$$\frac{\Vert \mathcal{A}^{(1)} \Vert_F \Vert \mathcal{C}^{(1)}*_v \mathcal{C} \Vert_F}{(1-\Vert \mathcal{E}*_t\mathcal{A}^{(1)} \Vert_F) \Vert  (\mathcal{C}*_2 \mathcal{A})^{(1)} *_2 \mathcal{C} \Vert_F}+1=2.8503.$$
As a result, the inequality presented in Theorem \ref{TheoremCInverse} is valid.
\end{example}

\begin{corollary}\label{Corollary2.1}
According to the condition of Theorem \ref{TheoremCInverse}, if $(\mathcal{C}*_s \mathcal{A})^{(1)}  = \mathcal{A}^{(1)}*_s \mathcal{C}^{(1)} $, the following inequality holds \\
$$\frac{\Vert \mathcal{D}^{(2)}_{*,\aunclfamily{\bf N}(\mathcal{C})}-\mathcal{A}^{(2)}_{*,\aunclfamily{\bf N}(\mathcal{C})}\Vert}{\Vert \mathcal{A}^{(2)}_{*,\aunclfamily{\bf N}(\mathcal{C})}\Vert} \leq \frac {1}{1-\Vert \mathcal{E}*_t\mathcal{A}^{(1)}\Vert}+1.$$
\end{corollary}

\begin{theorem}\label{TheoremBCInverse}
Let $\mathcal{A}, \mathcal{E}\in\mathbb{C}^{S(s)\times T(t)}$, $\mathcal{B}\in\mathbb{C}^{T(t)\times U(u)}$, $\mathcal{C}\in\mathbb{C}^{V(v)\times S(s)}$ and $\mathcal{A}^{(2)}_{\aunclfamily{\bf R}(\mathcal{B}),\aunclfamily{\bf N}(\mathcal{C})}$ be the $(\mathcal{B},\mathcal{C})$-inverse of $\mathcal{A}$.
If the subsequent conditions are satisfied for $\mathcal{D}=\mathcal{A}+\mathcal{E}$:\\
1. $((\mathcal{C}*_s\mathcal{D})^{(1)}*_v\mathcal{C}*_s\mathcal{D}*_t\mathcal{B}*_u \mathcal{B}^{(1)})^2=(\mathcal{C}*_s\mathcal{D})^{(1)}*_v\mathcal{C}*_s\mathcal{D}*_t\mathcal{B}*_u \mathcal{B}^{(1)}$; \\
2. $(\mathcal{C}^{(1)}*_v\mathcal{C}*_s\mathcal{D}*_t\mathcal{D}^{(1)})^2=\mathcal{C}^{(1)}*_v\mathcal{C}*_s\mathcal{D}*_t\mathcal{D}^{(1)}$;\\
3. $\rra{\mathcal{C}*_s\mathcal{A}*_t\mathcal{B}}=\rra{\mathcal{C}*_s\mathcal{D}*_t\mathcal{B}}=\rra{\mathcal{C}}=\rra{\mathcal{B}}$; \\
4. $\Vert \mathcal{A}^{(1)}\Vert \Vert \mathcal{E} \Vert <1 \ \  and \ \ \rra{\mathcal{A}}=\rra{\mathcal{D}},$\\
then it follows\\
\begin{equation}\label{The3.1}
\mathcal{D}^{(2)}_{\aunclfamily{\bf R}(\mathcal{B}),\aunclfamily{\bf N}(\mathcal{C})}=\mathcal{B}*_u \mathcal{B}^{(1)}*_t\mathcal{A}^{(1)}*_s \bigrho *_s\mathcal{C}^{(1)}*_v\mathcal{C},
\end{equation}
\begin{equation}\label{The3.30}
\frac{\Vert \mathcal{D}^{(2)}_{\aunclfamily{\bf R}(\mathcal{B}),\aunclfamily{\bf N}(\mathcal{C})}-\mathcal{A}^{(2)}_{\aunclfamily{\bf R}(\mathcal{B}),\aunclfamily{\bf N}(\mathcal{C})}\Vert}{\Vert \mathcal{A}^{(2)}_{\aunclfamily{\bf R}(\mathcal{B}),\aunclfamily{\bf N}(\mathcal{C})}\Vert} \leq
\frac{\Vert \mathcal{B}*_u  \mathcal{B}^{(1)}\Vert \Vert \mathcal{A}^{(1)} \Vert \Vert \mathcal{C}^{(1)}*_v \mathcal{C} \Vert}{(1-\Vert \mathcal{E}*_t\mathcal{A}^{(1)} \mathcal{E} \Vert) \Vert \mathcal{B}*_u  (\mathcal{C}*_s \mathcal{A}*_t \mathcal{B})^{(1)} *_v \mathcal{C} \Vert}+1.
\end{equation}
\end{theorem}

\begin{proof}
By \eqref{0.33}, the $(\mathcal{B},\mathcal{C})$-inverse of $\mathcal{D}$ is equal to
\begin{equation}\label{Proof3.1}
\mathcal{D}^{(2)}_{\aunclfamily{\bf R}(\mathcal{B}),\aunclfamily{\bf N}(\mathcal{C})}=\mathcal{B}*_u  (\mathcal{C}*_s \mathcal{D}*_t \mathcal{B} )^{(1)} *_v \mathcal{C}.
\end{equation}
Utilizing $\mathcal{D}=\mathcal{A}+\mathcal{E}$, it can be concluded
\begin{equation}\label{Proof3.2}
\mathcal{C}*_s \mathcal{D}*_t \mathcal{B} =\mathcal{C}*_s (\mathcal{A}+\mathcal{E}) *_t\mathcal{B}.
\end{equation}
Under the condition $\Vert \mathcal{A}^{(1)} \Vert \Vert \mathcal{E} \Vert <1 $, the inequality \eqref{0.4} guarantees invertibility of $\mathcal{I}+\mathcal{E}*_t\mathcal{A}^{(1)}$.
Next $\rra{\mathcal{A}}=\rra{\mathcal{D}}$, and using Theorem \ref{Chai1Inverse1Trank} the inner inverse of $\mathcal{D}=\mathcal{A}+\mathcal{E}$ is calculated as

\begin{equation}\label{Proof3.3}
\mathcal{D}^{(1)}=\mathcal{A}^{(1)}*_s \bigrho.
\end{equation}
By conditions 1,2 and \eqref{0.1}, substitute $\mathcal{D}^{(1)}$ into $(\mathcal{C}*_s \mathcal{D}*_t \mathcal{B} )^{(1)}$:
\begin{equation}\label{Proof3.4}
(\mathcal{C}*_s \mathcal{D}*_t \mathcal{B} )^{(1)}= \mathcal{B}^{(1)}*_t \mathcal{D}^{(1)}*_s \mathcal{C}^{(1)}.
\end{equation}
Thus,
\begin{equation}\label{Proof3.5}
\mathcal{D}^{(2)}_{\aunclfamily{\bf R}(\mathcal{B}),\aunclfamily{\bf N}(\mathcal{C})}= \mathcal{B}*_u  (\mathcal{B}^{(1)}*_t \mathcal{D}^{(1)}*_s \mathcal{C}^{(1)} )*_u  \mathcal{C} = \mathcal{B}*_u  \mathcal{B}^{(1)}*_t \mathcal{A}^{(1)}*_s \bigrho *_s \mathcal{C}^{(1)}*_v \mathcal{C}.
\end{equation}
From \eqref{0.33} and \eqref{Proof3.5}, $\mathcal{A}^{(2)}_{\aunclfamily{\bf R}(\mathcal{B}),\aunclfamily{\bf N}(\mathcal{C})}=\mathcal{B}*_u  (\mathcal{C}*_s \mathcal{A}*_t \mathcal{B})^{(1)} *_v \mathcal{C} $, and further

\begin{equation}\label{Proof3.6}
\mathcal{D}^{(2)}_{\aunclfamily{\bf R}(\mathcal{B}),\aunclfamily{\bf N}(\mathcal{C})}-\mathcal{A}^{(2)}_{\aunclfamily{\bf R}(\mathcal{B}),\aunclfamily{\bf N}(\mathcal{C})}=
\mathcal{B}*_u  \mathcal{B}^{(1)}*_t \mathcal{A}^{(1)}*_s \bigrho *_s \mathcal{C}^{(1)}*_v \mathcal{C} - \mathcal{B}*_u  (\mathcal{C}*_s \mathcal{A}*_t \mathcal{B} )^{(1)} *_v \mathcal{C}.
\end{equation}
An application of the norm to the expressions in \eqref{Proof3.6} yields
\begin{equation}\label{Proof3.7}
\Vert \mathcal{D}^{(2)}_{\aunclfamily{\bf R}(\mathcal{B}),\aunclfamily{\bf N}(\mathcal{C})}-\mathcal{A}^{(2)}_{\aunclfamily{\bf R}(\mathcal{B}),\aunclfamily{\bf N}(\mathcal{C})} \Vert \leq
\Vert \mathcal{B}*_u  \mathcal{B}^{(1)} \Vert \Vert \mathcal{A}^{(1)}\Vert \Vert \bigrho \Vert \Vert \mathcal{C}^{(1)}*_v \mathcal{C}\Vert + \Vert \mathcal{B}*_u  (\mathcal{C}*_s \mathcal{A}*_t \mathcal{B} )^{(1)} *_v \mathcal{C} \Vert.
\end{equation}
Further, \eqref{0.4} leads to
\begin{equation}\label{Proof3.8}
\Vert \bigrho\Vert \leq \frac{1}{1-\Vert \mathcal{E}*_t\mathcal{A}^{(1)}\Vert}.
\end{equation}
An application of \eqref{Proof3.8} into \eqref{Proof3.7} produces
\begin{equation}\label{Proof3.9}
\Vert \mathcal{D}^{(2)}_{\aunclfamily{\bf R}(\mathcal{B}),\aunclfamily{\bf N}(\mathcal{C})}-\mathcal{A}^{(2)}_{\aunclfamily{\bf R}(\mathcal{B}),\aunclfamily{\bf N}(\mathcal{C})} \Vert \leq \frac{\Vert \mathcal{B}*_u  \mathcal{B}^{(1)} \Vert \Vert \mathcal{A}^{(1)} \Vert \Vert \mathcal{C}^{(1)}*_v \mathcal{C} \Vert}{1-\Vert \mathcal{E}*_t\mathcal{A}^{(1)} \Vert}+\Vert \mathcal{B}*_u  (\mathcal{C}*_s \mathcal{A}*_t \mathcal{B} )^{(1)} *_v \mathcal{C} \Vert.
\end{equation}
The normalization of \eqref{Proof3.9} by $\Vert \mathcal{A}^{(2)}_{\aunclfamily{\bf R}(\mathcal{B}),\aunclfamily{\bf N}(\mathcal{C})}\Vert $ leads to
\begin{equation}\label{Proof3.10}
\frac{\Vert \mathcal{D}^{(2)}_{\aunclfamily{\bf R}(\mathcal{B}),\aunclfamily{\bf N}(\mathcal{C})}-\mathcal{A}^{(2)}_{\aunclfamily{\bf R}(\mathcal{B}),\aunclfamily{\bf N}(\mathcal{C})}\Vert}{\Vert \mathcal{A}^{(2)}_{\aunclfamily{\bf R}(\mathcal{B}),\aunclfamily{\bf N}(\mathcal{C})}\Vert}
\leq \frac{\Vert \mathcal{B}*_u  \mathcal{B}^{(1)}\Vert \Vert \mathcal{A}^{(1)} \Vert \Vert \mathcal{C}^{(1)}*_v \mathcal{C} \Vert}{(1-\Vert  \mathcal{E} *_t\mathcal{A}^{(1)}\Vert) \Vert \mathcal{B}*_u  (\mathcal{C}*_s \mathcal{A}*_t \mathcal{B})^{(1)} *_v \mathcal{C} \Vert}+1,
\end{equation}
which completes the proof.
\end{proof}

\begin{example}
Let $\mathcal{A}\in\mathbb{R}^{(2\times 2)\times (2\times 2)}$ be with elements
\begin{center}
$\mathcal{A}(:,:,1,1)=\begin{bmatrix} \ 1 & 0  \ \\ \ 0 & 0 \  \end{bmatrix}$;
$\mathcal{A}(:,:,1,2)=\begin{bmatrix} \ 0 & 1  \ \\ \ 1 & 0 \  \end{bmatrix}$;
$\mathcal{A}(:,:,2,1)=\begin{bmatrix} \ 1 & 0  \ \\ \ 1 & 0 \  \end{bmatrix}$;
$\mathcal{A}(:,:,2,2)=\begin{bmatrix} \ 0 & 0  \ \\ \ 0 & 0 \  \end{bmatrix}$.
\end{center}
Further, choose $\mathcal{B}\in\mathbb{R}^{(2\times 2)\times (2\times 2)}$ with components
 \begin{center}
$\mathcal{B}(:,:,1,1)=\begin{bmatrix} \ 3 & 0  \ \\ \ 0 & 0 \ \end{bmatrix}$;
$\mathcal{B}(:,:,1,2)=\begin{bmatrix} \ 0 & 1  \ \\ \ 1 & 0 \  \end{bmatrix}$;
$\mathcal{B}(:,:,2,1)=\begin{bmatrix} \ 1 & 0  \ \\ \ 1 & 0 \  \end{bmatrix}$;
$\mathcal{B}(:,:,2,2)=\begin{bmatrix} \ 0 & 0  \ \\ \ 0 & 0 \ \end{bmatrix}$,
\end{center}
as well as $\mathcal{C}\in\mathbb{R}^{(2\times 2)\times (2\times 2)}$ determined by
 \begin{center}
$\mathcal{C}(:,:,1,1)=\begin{bmatrix} \ 1 & 0  \ \\ \ 1 & 0 \ \end{bmatrix}$;\
$\mathcal{C}(:,:,1,2)=\begin{bmatrix} \ 0 & 1  \ \\ \ 0 & 0 \  \end{bmatrix}$;\
$\mathcal{C}(:,:,2,1)=\begin{bmatrix} \ 0 & 1  \ \\ \ 1 & 0 \  \end{bmatrix}$;\
$\mathcal{C}(:,:,2,2)=\begin{bmatrix} \ 0 & 0  \ \\ \ 0 & 0 \ \end{bmatrix}$.
\end{center}
Consider perturbation tensor $\mathcal{E}\in\mathbb{R}^{(2\times 2)\times (2\times 2)}$ defined by
\begin{center}
$\mathcal{E}(:,:,1,1)=\begin{bmatrix} \ 0.2 & 0 \ \\ \ 0 & 0 \ \end{bmatrix}$;\
$\mathcal{E}(:,:,1,2)=\begin{bmatrix} \ 0 & 0.2  \ \\ \ 0 & 0 \  \end{bmatrix}$;\
$\mathcal{E}(:,:,2,1)=\begin{bmatrix} \ 0 & 0  \ \\ \ 0.2 & 0 \  \end{bmatrix}$;\
$\mathcal{E}(:,:,2,2)=\begin{bmatrix} \ 0 & 0  \ \\ \ 0 & 0 \ \end{bmatrix}$.
\end{center}
Next, Algorithm 4 from \cite{LMA685} is applicable in computing $\mathcal{A}^{(1)}$.
The F-norm norms of $\mathcal{A}^{(1)}$ and $\mathcal{E}$ are calculated as

$$\Vert \mathcal{A}^{(1)} \Vert_F=2.4495, \ \ \ \Vert \mathcal{E} \Vert_F=0.3464, \ \ \ \Vert \mathcal{A}^{(1)}\Vert_F\Vert \mathcal{E} \Vert_F=0.8485<1.$$

Then,  Algorithm 3 from \cite{LMA685} can be used to verify $\rra{\mathcal{C}*_2\mathcal{A}*_2 \mathcal{B}}=\rra{\mathcal{C}*_2\mathcal{D}*_2\mathcal{B}}=\rra{\mathcal{B}}=\rra{\mathcal{C}}=3$ and
$\rra{\mathcal{A}}=\rra{\mathcal{D}}=3$. Indeed, after necessary calculations, one can verify the following representation of $ \mathcal{D}^{(2)}_{\aunclfamily{\bf R}(\mathcal{B}),\aunclfamily{\bf N}(\mathcal{C})}\in \mathbb{R}^{(2\times 2)\times(2\times 2)}$:
 \begin{center}
$\mathcal{D}^{(2)}_{\aunclfamily{\bf R}(\mathcal{B}),\aunclfamily{\bf N}(\mathcal{C})}(:,:,1,1)=\begin{bmatrix} \ 0.8333 & 0  \ \\ \ 0 & 0 \ \end{bmatrix}$;\
$\mathcal{D}^{(2)}_{\aunclfamily{\bf R}(\mathcal{B}),\aunclfamily{\bf N}(\mathcal{C})}(:,:,1,2)=\begin{bmatrix} \ 0.5787 & 0.8333  \ \\ \ -0.6944 & 0 \ \end{bmatrix}$;\\
$\mathcal{D}^{(2)}_{\aunclfamily{\bf R}(\mathcal{B}),\aunclfamily{\bf N}(\mathcal{C})}(:,:,2,1)=\begin{bmatrix} \ -0.6944 & 0  \ \\ \ 0.8333 & 0 \ \end{bmatrix}$;
$\mathcal{D}^{(2)}_{\aunclfamily{\bf R}(\mathcal{B}),\aunclfamily{\bf N}(\mathcal{C})}(:,:,2,2)=\begin{bmatrix} \ 0 & 0  \ \\ \ 0 & 0 \ \end{bmatrix}$.
 \end{center}
In addition, $ \mathcal{A}^{(2)}_{\aunclfamily{\bf R}(\mathcal{B}),\aunclfamily{\bf N}(\mathcal{C})}\in \mathbb{R}^{(2\times 2)\times(2\times 2)} $ includes components
  \begin{center}
$\mathcal{A}^{(2)}_{\aunclfamily{\bf R}(\mathcal{B}),\aunclfamily{\bf N}(\mathcal{C})}(:,:,1,1)=\begin{bmatrix} \ 1 & 0  \ \\ \ 0 & 0 \ \end{bmatrix}$;\
$\mathcal{A}^{(2)}_{\aunclfamily{\bf R}(\mathcal{B}),\aunclfamily{\bf N}(\mathcal{C})}(:,:,1,2)=\begin{bmatrix} \ 1 & 1  \ \\ \ -1 & 0 \ \end{bmatrix}$;\\
$\mathcal{A}^{(2)}_{\aunclfamily{\bf R}(\mathcal{B}),\aunclfamily{\bf N}(\mathcal{C})}(:,:,2,1)=\begin{bmatrix} \ -1 & 0  \ \\ \ 1 & 0 \ \end{bmatrix}$;\
$\mathcal{A}^{(2)}_{\aunclfamily{\bf R}(\mathcal{B}),\aunclfamily{\bf N}(\mathcal{C})}(:,:,2,2)=\begin{bmatrix} \ 0 & 0  \ \\ \ 0 & 0 \ \end{bmatrix}$.
 \end{center}
 By calculation, $(\mathcal{C}*_2\mathcal{D})^{(1)}*_2\mathcal{C}*_2\mathcal{D}*_2\mathcal{B}*_2\mathcal{B}^{(1)}$ and $\mathcal{C}^{(1)}*_2\mathcal{C}*_2\mathcal{D}*_2\mathcal{D}^{(1)}$ are idempotent. Conditions are satisfied, and calculation gives
  \begin{center}
 $\Vert \mathcal{D}^{(2)}_{\aunclfamily{\bf R}(\mathcal{B}),\aunclfamily{\bf N}(\mathcal{C})}-\mathcal{A}^{(2)}_{\aunclfamily{\bf R}(\mathcal{B}),\aunclfamily{\bf N}(\mathcal{C})}\Vert_F=0.6690$, \ \ \
  $\Vert \mathcal{C}^{(1)}*_2 \mathcal{C} \Vert_F=1.7321$,\ \ \
 $\Vert \mathcal{B}*_2 \mathcal{B}^{(1)}\Vert_F=1.7321$, \\
 $\Vert \mathcal{E}*_2\mathcal{A}^{(1)} \Vert_F=0.4899$, \ \ \
  $\Vert \mathcal{A}^{(2)}_{\aunclfamily{\bf R}(\mathcal{B}),\aunclfamily{\bf N}(\mathcal{C})}\Vert_F=\Vert \mathcal{B}*_2 (\mathcal{C}*_2 \mathcal{A}*_2 \mathcal{B})^{(1)} *_2 \mathcal{C} \Vert_F=2.4495$.
 \end{center}
Then
$$\frac{\Vert \mathcal{D}^{(2)}_{\aunclfamily{\bf R}(\mathcal{B}),\aunclfamily{\bf N}(\mathcal{C})}-\mathcal{A}^{(2)}_{\aunclfamily{\bf R}(\mathcal{B}),\aunclfamily{\bf N}(\mathcal{C})}\Vert}{\Vert \mathcal{A}^{(2)}_{\aunclfamily{\bf R}(\mathcal{B}),\aunclfamily{\bf N}(\mathcal{C})}\Vert}=0.2731$$
$$\frac{\Vert \mathcal{B}*_2 \mathcal{B}^{(1)}\Vert_F \Vert \mathcal{A}^{(1)} \Vert_F \Vert \mathcal{C}^{(1)}*_2 \mathcal{C} \Vert_F}{(1-\Vert  \mathcal{E} *_2\mathcal{A}^{(1)}\Vert_F) \Vert \mathcal{B}*_2 (\mathcal{C}*_2 \mathcal{A}*_2 \mathcal{B})^{(1)} *_2 \mathcal{C} \Vert_F}+1=6.8812.$$
Therefore, the inequality of Theorem \ref{TheoremBCInverse} holds.
\end{example}

\begin{corollary}
The conditions 1 and 2 in Theorem 2.1 can be changed to \\
 $1$. $(\mathcal{C}^{(1)}*_v \mathcal{C}*_s \mathcal{D}*_t \mathcal{B}*_u  (\mathcal{D}*_t \mathcal{B} )^{(1)}  )^2=\mathcal{C}^{(1)}*_v \mathcal{C}*_s \mathcal{D}*_t \mathcal{B}*_u  (\mathcal{D}*_t \mathcal{B})^{(1)},$ \\
$2$. $(\mathcal{D}^{(1)}*_s \mathcal{D}*_t \mathcal{B}*_u  \mathcal{B}^{(1)} )^2=\mathcal{D}^{(1)}*_s \mathcal{D}*_t \mathcal{B}*_u  \mathcal{B}^{(1)} $\\
 and the conclusion remains unchanged.
\end{corollary}

\begin{proof}
In order to ensure the condition \eqref{Proof3.4}, replace conditions 1,2 in Theorem \ref{TheoremBCInverse} with new conditions  1.2, and follow the subsequent proof as that in Theorem \ref{TheoremBCInverse}.
\end{proof}

\begin{corollary}
Based on the condition of Theorem \ref{TheoremBCInverse}, if $(\mathcal{C}*_s \mathcal{A}*_t \mathcal{B} )^{(1)}  =\mathcal{B}^{(1)}*_t \mathcal{A}^{(1)}*_s \mathcal{C}^{(1)} $, then \\
$$\frac{\Vert \mathcal{D}^{(2)}_{\aunclfamily{\bf R}(\mathcal{B}),\aunclfamily{\bf N}(\mathcal{C})}-\mathcal{A}^{(2)}_{\aunclfamily{\bf R}(\mathcal{B}),\aunclfamily{\bf N}(\mathcal{C})}\Vert}{\Vert \mathcal{A}^{(2)}_{\aunclfamily{\bf R}(\mathcal{B}),\aunclfamily{\bf N}(\mathcal{C})}\Vert} \leq \frac {1}{1-\Vert \mathcal{E}*_t\mathcal{A}^{(1)} \Vert}+1.$$
\end{corollary}

\section{Conclusion}
Three perturbation problems involving tensors are examined.
A small perturbation $\mathcal{E}$ in $\mathcal{A}$ is considered, such that $\|\mathcal{A}^{(1)}\|\|\mathcal{E}\| < 1$ in all cases.
The research question aims to understand how $\mathcal{E}$ impacts the computation of generalized inverses of $\mathcal{A}$.

The first research task investigates relationships between $ (\mathcal{A} + \mathcal{E})^{(1)}$ and $\mathcal{A}^{(1)}$ in terms of the multiplicative factors denoted as
$\bigrho =(\mathcal{I} + \mathcal{E}*_t \mathcal{A}^{(1)})^{-1}$ and $\bigdelta =(\mathcal{I} + \mathcal{A}^{(1)}*_s\mathcal{E})^{-1}$.
The objective is to explore the relationships represented by the equation $(\mathcal{A} + \mathcal{E})^{(1)}=\mathcal{A}^{(1)} *_s\bigrho=\bigdelta *_t\mathcal{A}^{(1)}$ under certain conditions.

We consider a small perturbation satisfying some specified requirements in the second research task.
The research task aims to investigate relationship between $ (\mathcal{A} + \mathcal{E})^{(2)}$ and $\mathcal{A}^{(2)}$, which is defined using the multiplicative factors
$\bigDdelta   =(\mathcal{I} + \mathcal{A}^{(2)}*_s\mathcal{E})^{-1}$ and $\bigDrho  =(\mathcal{I} + \mathcal{E}*_t \mathcal{A}^{(2)})^{-1}$.
The goal is to explore the relationship of the form $ (\mathcal{A} + \mathcal{E})^{(2)}=\mathcal{A}^{(2)} *_s\bigDrho =\bigDdelta   *_t\mathcal{A}^{(2)}$.

The third research task aims to generalize the previous two tasks regarding outer inverses with determined ranges or/and null spaces.

The relationship between $\mathcal{E}$ and generalized inverses is a multifaceted challenge.
Our research has specified particular conditions that allow us to evaluate this relationship effectively through various expressions and inequalities.
The effectiveness of perturbation analysis often depends on the validity of its basic assumptions, such as the small size of perturbations.
If these assumptions are not met, the conclusions drawn may be misleading.
Different constraints will inevitably yield diverse results, emphasizing the importance of carefully considering restrictions and choices.

\noindent {\bf Funding.}
The first author is supported by the National Natural Science Foundation of China (NSFC) (No. 11901079), and China Postdoctoral Science Foundation (No. 2021M700751), and the Scientific and Technological Research Program Foundation of Jilin Province, (No. JJKH20190690KJ; No. JJKH20220091KJ; No. JJKH20250851KJ).

The third and fourth authors are supported by the  Ministry of Science, Technological Development and Innovation, Republic of Serbia (No. 451-03-137/2025-03/ 200124).

Predrag Stanimirovi\'c is supported by the Ministry of Science and Technology of China under grant  H20240841.

\medskip
\noindent {\bf Conflict of Interest.}
The authors declare that they have no potential conflict of interest.

\section*{Data Availability Statement}
Data available on request from the authors.
%\medskip
%\noindent {\bf ORCID.}
%Daochang Zhang \orcidC \href{https://orcid.org/0000-0002-9648-362X}{ \hspace{2mm}\textcolor{lightblue}{https://orcid.org/0000-0002-9648-362X}} \\
%Jingqian Li \orcidC \href{https://orcid.org/0009-0006-3763-7867}{ \hspace{2mm}\textcolor{lightblue}{https://orcid.org/0009-0006-3763-7867}} \\
%Dijana Mosi\'c \orcidC \href{https://orcid.org/0000-0002-3255-9322}{ \hspace{2mm}\textcolor{lightblue}{https://orcid.org/0000-0001-6104-5171}} \\
%Predrag S. Stanimirovi\'c \orcidA \href{https://orcid.org/0000-0003-0655-3741}{ \hspace{2mm}\textcolor{lightblue}{ https://orcid.org/0000-0003-0655-3741}}\\

\bibliographystyle{abbrv}

\end{document}